\documentclass[hidelinks,onefignum,onetabnum]{siamart250211}

\usepackage{lipsum}
\usepackage{amsfonts}
\usepackage{graphicx}
\usepackage{epstopdf}
\usepackage{algorithmic}
\ifpdf
  \DeclareGraphicsExtensions{.eps,.pdf,.png,.jpg}
\else
  \DeclareGraphicsExtensions{.eps}
\fi

\newsiamremark{remark}{Remark}
\newsiamremark{hypothesis}{Hypothesis}
\crefname{hypothesis}{Hypothesis}{Hypotheses}
\newsiamthm{claim}{Claim}
\newsiamremark{fact}{Fact}
\crefname{fact}{Fact}{Facts}

\headers{ARB Preconditioners for Large-Scale Parametrized PDE\lowercase{s}}{X. F\lowercase{an}, Q. H\lowercase{u}, S. Z\lowercase{hang} }

\title{Additive Reduced Basis Preconditioners for Large-Scale Parametrized PDE\lowercase{s}\thanks{This manuscript is available as a preprint.
\funding{This work was supported in part by National Natural Science Foundation of China
(U25A20200).}}}

\author{
X\lowercase{innan} F\lowercase{an}\thanks{College of Mathematics, Sichuan University, 610064 Chengdu, China (\email{fanxinnan2025@163.com}).}
\and Q\lowercase{ixiao} H\lowercase{u}\thanks{Key Laboratory of Nuclear Reactor System Design Technology, Nuclear Power Institute of China,   610213 Chengdu, China (\email{hqixiao2025@163.com}).}
\and S\lowercase{hiquan} Z\lowercase{hang}\thanks{ College of Mathematics, Sichuan University, 610064 Chengdu, China
(\url{shiquanzhang@scu.edu.cn}).}
}

\usepackage{amsopn}

\ifpdf
\hypersetup{
  pdftitle={Additive Reduced Basis Preconditioners for Large-Scale Parametrized PDEs},
pdfauthor={author}
}
\fi

\usepackage{pgfplots}
\pgfplotsset{compat=1.17}
\usepackage{array}
\usepackage{graphicx,epstopdf}
\usepackage[caption=false]{subfig}

\begin{document}

\maketitle

\begin{abstract}
We introduce a class of additive reduced basis preconditioners designed to accelerate the iterative solution of large-scale linear systems arising from discretized parametrized PDEs. The main idea is to regularize the inherently singular reduced-order approximation by adding simple correction terms: either a scaled identity correction or a projected correction based on a basic preconditioner. This yields nonsingular preconditioners under explicit and easily checked conditions. The construction and application of the preconditioners are integrated into an FGMRES framework through an offline strategy that dynamically builds the reduced-basis component by proper orthogonal decomposition at each FGMRES step. We establish sufficient conditions for the nonsingularity of the preconditioners and derive error bounds for the preconditioned Richardson iteration. Numerical results for convection-diffusion, anisotropic vortex, Stokes, and Helmholtz problems are provided to verify the efficiency and convergence of the proposed ARB preconditioners. The method consistently converges in a few iterations and substantially reduces online solve time, supporting its efficiency for multi-query engineering scenarios.
\end{abstract}

\section{Introduction}
Parametrized partial differential equations (PDEs) involving varying physical, geometric, or operating conditions play a pivotal role in modern scientific computing and engineering design. In digital twins, uncertainty quantification, design optimization, and real-time control, one must repeatedly solve the same model for many different parameter configurations \cite{benner2015survey, kapteyn2021probabilistic, prudhomme2002reliable}. High-fidelity finite element and finite volume discretizations provide the accuracy needed for engineering validation, but they lead to large sparse linear systems whose size can reach millions or billions of degrees of freedom \cite{elman2014finite, wathen2015preconditioning}. Direct solvers are therefore often too expensive in the many-query setting, and Krylov subspace methods, including the generalized minimal residual (GMRES) method and its flexible variants, have become standard tools for large nonsymmetric or indefinite systems \cite{saad1993flexible, saad2003iterative, saad1986gmres}. Their efficiency, however, depends strongly on the condition number and spectral structure of the discretized operators; as the mesh is refined or as parameters move across physical regimes, convergence can deteriorate severely \cite{carr2021preconditioning, simoncini2007recent}.

Model order reduction offers a complementary way to reduce this computational burden. Reduced basis (RB) and proper orthogonal decomposition (POD) methods exploit the low intrinsic dimension of many parametrized solution manifolds by projecting the full-order problem onto a subspace generated from selected high-fidelity snapshots \cite{hesthaven2016certified, quarteroni2015reduced, rozza2008reduced}. The offline-online decomposition and rigorous a posteriori estimation machinery make this paradigm especially effective for rapid and reliable repeated evaluations \cite{prudhomme2002reliable, grepl2005posteriori}. For non-affine or nonlinear parameter dependence, empirical interpolation and discrete empirical interpolation are commonly used to recover an efficient online stage \cite{barrault2004empirical, chaturantabut2010nonlinear}. Matrix DEIM and analytical preconditioning ideas further extend this strategy to algebraic operators and parametric collocation settings \cite{chen2014parametric, negri2015efficient}.

Nevertheless, directly using a reduced model as the final solver can be limiting when the solution manifold is not well approximated by a low-dimensional space. In particular, transport-dominated solutions and wave propagation can exhibit slow Kolmogorov-width decay, so a small static reduced space can fail to meet strict engineering tolerances \cite{dahmen2014double, greif2019decay}. Although local nonlinear reduced bases may mitigate this limitation \cite{amsallem2012nonlinear}, achieving the required accuracy can still require a much larger reduced space and a more expensive offline stage, which weakens the computational advantage of direct reduced-order solves.

Rather than replacing the high-fidelity solver directly by a reduced model, a growing line of work uses low-dimensional information to accelerate the full-order iterative process. Krylov deflation and subspace recycling reuse spectral or solution information across related linear systems \cite{morgan2002gmres, parks2006recycling}. POD-augmented iterations provide a closely related way to inject snapshot information into Krylov solvers \cite{carlberg2016krylov}. Domain decomposition and two-level preconditioning provide another classical framework in which a coarse component removes global error modes while local or algebraic preconditioners handle the remaining fine-scale components \cite{dolean2015domain, wathen2015preconditioning}. Recent data-driven work also explores preconditioning ideas in physics-informed neural networks and learning-based domain-decomposition strategies \cite{liu2024preconditioning, sun2024domain}. Projection-based RB preconditioning is attractive in this landscape because it keeps the full-order residual as the convergence criterion while exploiting parametrized snapshot information in a controlled algebraic manner.

Recently, the Multi Space Reduced Basis (MSRB) preconditioner was introduced for large-scale parametrized PDEs and further investigated for advection-diffusion and Stokes problems \cite{dalsanto2017investigation, santo2018multi, dalsanto2019stokes}. The MSRB approach constructs a sequence of RB spaces, each tailored to an error-residual equation arising at a given step of Richardson or flexible GMRES (FGMRES), and combines an RB coarse solver with a nonsingular fine-grid preconditioner. Following a similar broad philosophy of using iteration-dependent reduced spaces \cite{santo2018multi}, we propose a class of reduced basis preconditioners to accelerate the iterative solution of large-scale linear systems arising from discretized parametrized PDEs. Instead of relying on a multiplicative two-level structure with a separate fine-grid preconditioner \cite{santo2018multi}, the key difference is that we directly add simple correction terms to the singular RB approximation. This leads to additive reduced basis (ARB) preconditioners with explicit nonsingularity conditions, inexpensive application formulas, and a compact implementation. The resulting preconditioners are combined with FGMRES for solving the full-order model, so the final accuracy is that of the high-fidelity discretization rather than that of the RB approximation alone.

The remainder of this paper is organized as follows. \Cref{sec:preliminaries} recalls the RB method and the GMRES framework. \Cref{sec:arb-preconditioners} constructs the ARB preconditioners, proves their nonsingularity under transparent assumptions, and explains their dynamic offline generation for FGMRES. \Cref{sec:numerical-experiments} presents numerical experiments for convection-diffusion, anisotropic vortex, Stokes, and Helmholtz problems, with comparisons against classical ILU and AMG preconditioners. Finally, \cref{sec:conclusions} summarizes the paper and discusses the main practical implications.

\section{Preliminaries}\label{sec:preliminaries}
Solving parametrized PDEs is computationally demanding, especially when many parameter instances must be considered. After discretizing such equations with standard methods such as the Galerkin finite element method on a subspace $V_h$ of dimension $N_h$ \cite{brezzi1991mixed, elman2014finite}, we obtain a parametrized linear system of the form
\begin{equation}\label{eq:parametrized-linear-system}
\mathbf{A}_h(\boldsymbol{\mu})\mathbf{u}_h(\boldsymbol{\mu})=\mathbf{f}_h(\boldsymbol{\mu}),
\end{equation}
where $\boldsymbol{\mu}\in\mathcal{D}\subset\mathbb{R}^{m}$ is the vector of input parameters, $\mathbf{u}_h(\boldsymbol{\mu}), \mathbf{f}_h(\boldsymbol{\mu})\in\mathbb{R}^{N_h}$ denote the solution vector and load vector, respectively, and $\mathbf{A}_h(\boldsymbol{\mu})\in\mathbb{R}^{N_h\times N_h}$ is the stiffness matrix, assumed nonsingular on the working space.

\subsection{The RB method}

The RB method aims to approximate the solution manifold of \eqref{eq:parametrized-linear-system}, denoted by $\mathcal{M}_h = \{\mathbf{u}_h(\boldsymbol{\mu}) : \boldsymbol{\mu} \in \mathcal{D}\}$, using a low-dimensional subspace $V_N$ of dimension $N$, with $N \ll N_h$. For a comprehensive theoretical foundation and derivation of this method, we refer the reader to \cite[chapter 3]{hesthaven2016certified}. The method approximates the high-fidelity solution $\mathbf{u}_h(\boldsymbol{\mu})$ in the reduced subspace $V_N$ by a linear combination of basis vectors,
\begin{displaymath}
        \mathbf{u}_h(\boldsymbol{\mu}) \approx \mathbf{P}\mathbf{u}_N(\boldsymbol{\mu}),
\end{displaymath}
where $\mathbf{P} \in \mathbb{R}^{N_h \times N}$ is a matrix whose columns, $\{\mathbf{p}_n\}_{n=1}^N$, form an orthonormal basis of $V_N$. Projecting the original system \eqref{eq:parametrized-linear-system} onto $V_N$ by the Galerkin method yields the reduced system for $\mathbf{u}_N(\boldsymbol{\mu})$:
\begin{equation}\label{eq:reduced-system}
\mathbf{A}_N(\boldsymbol{\mu})\mathbf{u}_N(\boldsymbol{\mu})=\mathbf{f}_N(\boldsymbol{\mu}),
\end{equation}
where
\begin{displaymath}
        \mathbf{A}_N(\boldsymbol{\mu}) = \mathbf{P}^T \mathbf{A}_h(\boldsymbol{\mu}) \mathbf{P} \in \mathbb{R}^{N \times N}, \quad \mathbf{f}_N(\boldsymbol{\mu}) = \mathbf{P}^T \mathbf{f}_h(\boldsymbol{\mu}) \in \mathbb{R}^{N}.
\end{displaymath}
Provided $\mathbf{A}_N(\boldsymbol{\mu})$ is nonsingular, the reduced coefficients are recovered by $\mathbf{u}_N(\boldsymbol{\mu}) = \mathbf{A}_N^{-1}(\boldsymbol{\mu})\mathbf{f}_N(\boldsymbol{\mu})$. The efficiency of this method relies on an offline-online decomposition (see \cite[subsection 3.3]{hesthaven2016certified}), in which all computations whose costs depend on $N_h$ are performed offline and the online cost depends only on $N$. When $\mathbf{A}_h(\boldsymbol{\mu})$ and $\mathbf{f}_h(\boldsymbol{\mu})$ depend affinely on $\boldsymbol{\mu}$, the reduced system \eqref{eq:reduced-system} can be assembled and solved rapidly online. General non-affine or nonlinear dependencies are commonly treated by empirical interpolation or related hyper-reduction techniques \cite{barrault2004empirical, chaturantabut2010nonlinear}.

\begin{algorithm}[t]
    \caption{POD}
    \label{alg:pod}
    \begin{algorithmic}[1]
   \STATE \textbf{procedure} $\text{POD}\left( \mathbf{u}_h(\boldsymbol{\mu}_1), \mathbf{u}_h(\boldsymbol{\mu}_2), \ldots, \mathbf{u}_h(\boldsymbol{\mu}_{n_s}); \epsilon_{\text{tol}}\right)$
        \STATE $\mathbf{S}=\left[\mathbf{u}_h(\boldsymbol{\mu}_1), \mathbf{u}_h(\boldsymbol{\mu}_2), \ldots, \mathbf{u}_h(\boldsymbol{\mu}_{n_s})\right]$
        \STATE Compute the economy SVD: $[\mathbf{U}_r, \boldsymbol{\Sigma}_r, \mathbf{Z}_r] = \text{SVD}(\mathbf{S})$
        \STATE Extract the singular values $\{\sigma_i\}_{i=1}^{r}$ from $\boldsymbol{\Sigma}_r$
        \STATE Select the smallest $N$ such that $\frac{\sum_{i=1}^{N} \sigma_i^2}{\sum_{i=1}^{r} \sigma_i^2} \ge 1 - \epsilon_{\text{tol}}^2$

        \STATE Save $\mathbf{P}=[(\mathbf{U}_r)_{:,1}, \dots, (\mathbf{U}_r)_{:,N}]$
       \STATE \textbf{end procedure}
    \end{algorithmic}
\end{algorithm}

The quality of the RB approximation depends critically on the construction of the basis matrix $\mathbf{P}$. A widely used approach is POD, discussed in detail in \cite[subsection 3.2]{hesthaven2016certified}. POD generates a basis that minimizes the average projection error over a training set of snapshots. Let $\Xi_{\text{train}} = \{\boldsymbol{\mu}_1, \dots, \boldsymbol{\mu}_{n_s}\} \subset \mathcal{D}$ be a finite training set. We first compute the high-fidelity solutions (snapshots) for these parameters and assemble them into a matrix
\begin{displaymath}
\mathbf{S}=\left[\mathbf{u}_h(\boldsymbol{\mu}_1), \mathbf{u}_h(\boldsymbol{\mu}_2), \ldots, \mathbf{u}_h(\boldsymbol{\mu}_{n_s})\right].
\end{displaymath}
The POD basis is derived from the economy singular value decomposition $\mathbf{S} = \mathbf{U}_r\mathbf{\Sigma}_r\mathbf{Z}_r^T$, where $r=\operatorname{rank}(\mathbf{S})$. Here, $\mathbf{U}_r \in \mathbb{R}^{N_h \times r}$ contains the left singular vectors, and $\mathbf{\Sigma}_r = \operatorname{diag}(\sigma_1, \dots, \sigma_r)$ contains the singular values sorted in descending order. The optimal basis of dimension $N$ consists of the first $N$ columns of $\mathbf{U}_r$. The approximation error is governed by the decay of the singular values
\begin{displaymath}
        \sum_{i=1}^{n_s} \left\| \mathbf{u}_h(\boldsymbol{\mu}_i) - \mathbf{P}\mathbf{P}^T\mathbf{u}_h(\boldsymbol{\mu}_i) \right\|_2^2 = \sum_{i=N+1}^{r} \sigma_i^2.
\end{displaymath}
For more details, we refer the reader to \cite{quarteroni2015reduced}. The dimension $N$ is typically selected so that the relative energy of the neglected modes is at most $\epsilon_{\text{tol}}^2$. The detailed procedure is presented in \cref{alg:pod}.

\subsection{The GMRES method}
The GMRES method, proposed by Saad and Schultz \cite{saad1986gmres}, is widely regarded as one of the most robust Krylov subspace techniques for solving linear systems \cite{simoncini2007recent}. Consider a generic linear system of the form $\mathbf{A}\mathbf{x} = \mathbf{b}$, where $\mathbf{A} \in \mathbb{R}^{n\times n}$ is a large sparse nonsingular matrix and $\mathbf{b} \in \mathbb{R}^n$ is an $n$-dimensional vector. Let $\mathbf{x}_0$ be an initial guess and let $\mathbf{r}_0 = \mathbf{b} - \mathbf{A}\mathbf{x}_0$ be the initial residual. GMRES seeks an approximate solution $\mathbf{x}_k$ at the $k$-th iteration from the affine subspace $\mathbf{x}_0 + \mathcal{K}_k(\mathbf{A}, \mathbf{r}_0)$, where $\mathcal{K}_k$ is the $k$-th-order Krylov subspace defined by
\begin{displaymath}
    \mathcal{K}_k(\mathbf{A}, \mathbf{r}_0) = \text{span}\{\mathbf{r}_0, \mathbf{A}\mathbf{r}_0, \dots, \mathbf{A}^{k-1}\mathbf{r}_0\}.
\end{displaymath}
The approximation $\mathbf{x}_k$ is uniquely determined by minimizing the Euclidean norm of the residual vector $\mathbf{r}_k = \mathbf{b} - \mathbf{A}\mathbf{x}_k$ over the affine subspace, that is,
\begin{displaymath}
        \mathbf{x}_k = \text{argmin}_{\mathbf{x} \in \mathbf{x}_0 + \mathcal{K}_k(\mathbf{A}, \mathbf{r}_0)} \|\mathbf{b} - \mathbf{A}\mathbf{x}\|_2.
\end{displaymath}
To solve this minimization problem efficiently, GMRES uses the Arnoldi process to construct an orthonormal basis $\{\mathbf{v}_1, \dots, \mathbf{v}_k\}$ for $\mathcal{K}_k(\mathbf{A}, \mathbf{r}_0)$. Let $\mathbf{V}_k = [\mathbf{v}_1, \dots, \mathbf{v}_k] \in \mathbb{R}^{n \times k}$. The relation between the matrix $\mathbf{A}$ and the Arnoldi vectors is governed by
\begin{displaymath}
        \mathbf{A}\mathbf{V}_k = \mathbf{V}_{k+1}\overline{\mathbf{H}}_k,
\end{displaymath}
where $\overline{\mathbf{H}}_k \in \mathbb{R}^{(k+1) \times k}$ is an upper Hessenberg matrix generated during the Arnoldi process. Writing $\mathbf{x}_k = \mathbf{x}_0 + \mathbf{V}_k\mathbf{y}_k$ with $\mathbf{y}_k \in \mathbb{R}^k$, the minimization problem reduces to a low-dimensional least squares problem,
\begin{displaymath}
        \mathbf{y}_k = \text{argmin}_{\mathbf{y} \in \mathbb{R}^k} \left\| \beta \mathbf{e}_1 - \overline{\mathbf{H}}_k \mathbf{y} \right\|_2,
\end{displaymath}
where $\beta = \|\mathbf{r}_0\|_2$ and $\mathbf{e}_1 = [1, 0, \dots, 0]^T \in \mathbb{R}^{k+1}$. This reduced problem is typically solved via QR factorization using Givens rotations. For details of the procedure, which uses the modified Gram-Schmidt algorithm for numerical stability, see \cite[chapter 6]{saad2003iterative}.
As the iteration count $k$ increases, the storage requirements for $\mathbf{V}_k$ and the computational cost of the orthogonalization grow linearly and quadratically, respectively. To mitigate this, a restarted version, denoted as GMRES($m$), is often employed \cite[chapter 6]{saad2003iterative}. In GMRES($m$), the algorithm is restarted every $m$ steps, using the most recent approximation $\mathbf{x}_m$ as the new initial guess $\mathbf{x}_0$, thereby keeping the Krylov subspace dimension bounded.

\begin{remark}
Using either the GMRES method or the RB method alone is often inefficient for parametrized PDEs. GMRES may converge slowly unless a problem-dependent preconditioner is available. The RB method often incurs high offline costs when high accuracy is required, especially for problems with complicated parameter dependence. Consequently, we employ the RB method as a preconditioning component for GMRES, which can reduce the offline burden while accelerating the iterative scheme to the desired accuracy.
\end{remark}

\section{ARB preconditioners}\label{sec:arb-preconditioners}

Because standard GMRES may converge slowly, preconditioning is used to accelerate convergence. With a nonsingular preconditioner $\mathbf{M}$, the preconditioned linear system of \cref{eq:parametrized-linear-system} becomes
\begin{equation}\label{eq:preconditioned}
\mathbf{M}^{-1}\mathbf{A}_h(\boldsymbol{\mu})\mathbf{u}_h(\boldsymbol{\mu})=\mathbf{M}^{-1}\mathbf{f}_h(\boldsymbol{\mu}).
\end{equation}
A good preconditioner requires $\mathbf{M}^{-1}$ to be a useful approximation of $\mathbf{A}_h^{-1}(\boldsymbol{\mu})$, while its application is much cheaper than solving with $\mathbf{A}_h(\boldsymbol{\mu})$ directly.

In this section, we construct ARB preconditioners to accelerate the solution of \cref{eq:parametrized-linear-system}. We first present the basic construction and analyze nonsingularity. We then explain why dynamic construction is needed, use Richardson's method to illustrate the theoretical results, and finally formulate the complete algorithms proposed in this work.

\subsection{Construction of ARB preconditioners}
From \eqref{eq:reduced-system}, the RB solution can be written as
\begin{displaymath}
    \mathbf{P}\mathbf{u}_N(\boldsymbol{\mu})=\mathbf{P}\left(\mathbf{P}^T\mathbf{A}_h(\boldsymbol{\mu})\mathbf{P}\right)^{-1}\mathbf{P}^T\mathbf{f}_h(\boldsymbol{\mu}):=\widetilde{\mathbf{M}}(\boldsymbol{\mu})\mathbf{f}_h(\boldsymbol{\mu}),
 \end{displaymath}
where
 \begin{displaymath}
    \widetilde{\mathbf{M}}(\boldsymbol{\mu}):=\mathbf{P}\left(\mathbf{P}^T\mathbf{A}_h(\boldsymbol{\mu})\mathbf{P}\right)^{-1}\mathbf{P}^T.
 \end{displaymath}
Compared with the full-order solution of \eqref{eq:parametrized-linear-system}, i.e.,
\begin{displaymath}
    \mathbf{u}_h(\boldsymbol{\mu})=\mathbf{A}_h^{-1}(\boldsymbol{\mu})\mathbf{f}_h(\boldsymbol{\mu}),
 \end{displaymath}
this suggests that the RB method uses $\widetilde{\mathbf{M}}(\boldsymbol{\mu})$ as an approximation of $\mathbf{A}_h^{-1}(\boldsymbol{\mu})$. Consequently, one might expect the inverse of $\widetilde{\mathbf{M}}(\boldsymbol{\mu})$ to be a useful preconditioner for \eqref{eq:parametrized-linear-system}. Unfortunately, $\widetilde{\mathbf{M}}(\boldsymbol{\mu})$ is always singular, which prevents us from using it directly as a preconditioner. The idea of ARB is to add a simple term to $\widetilde{\mathbf{M}}(\boldsymbol{\mu})$, i.e.,
\begin{displaymath}
    \mathbf{M}^*(\boldsymbol{\mu}):=\widetilde{\mathbf{M}}(\boldsymbol{\mu})+\mathbf{C}
\end{displaymath}
such that $\mathbf{M}^*(\boldsymbol{\mu})$ is nonsingular. The simplest case is based on the following proposition.
\begin{proposition}\label{prop:shifted-nonsingularity}
    For any matrix $\mathbf{A}\in\mathbb{R}^{n\times n}$, there exists $\alpha\in \mathbb{R}$ such that $\mathbf{A}+\alpha\mathbf{I}$ is nonsingular, where $\mathbf{I}\in\mathbb{R}^{n\times n}$ is the identity matrix.
\end{proposition}
\begin{proof}
    Let $\sigma(\mathbf{A})$ be the spectrum of $\mathbf{A}$, i.e.,
\begin{displaymath}
    \sigma(\mathbf{A})=\{\lambda\in\mathbb{C}: \mathbf{A}-\lambda\mathbf{I}~\text{is singular}\}.
\end{displaymath}
By the fundamental theorem of algebra, $\sigma(\mathbf{A})$ is finite. Fix $\alpha_0\in\mathbb{R}$ and consider
\begin{displaymath}
    \alpha_t=\alpha_0+\frac{1}{t},\quad t=1,2,\ldots.
\end{displaymath}
Since $\{\alpha_t\}$ is infinite, there exists $t_0$ such that $-\alpha_{t_0}\notin\sigma(\mathbf{A})$.
\end{proof}
In practical applications, we do not compute the spectrum of $\mathbf{A}$ because it is a large-scale matrix. For a fixed $\boldsymbol{\mu}$ or finitely many values of $\boldsymbol{\mu}$, the spectrum of $\widetilde{\mathbf{M}}(\boldsymbol{\mu})$ is a finite set. A simple practical idea is therefore to preselect $\alpha$ and directly construct the Type I preconditioner as
\begin{equation}\label{def:p1}
  \mathbf{M}_{(1)}(\boldsymbol{\mu}):=\left(\widetilde{\mathbf{M}}(\boldsymbol{\mu})+\alpha\mathbf{I}\right)^{-1}.
\end{equation}
According to \cref{prop:shifted-nonsingularity}, there exists $\alpha$ such that $\mathbf{M}_{(1)}(\boldsymbol{\mu})$ is well defined. More precisely, this holds whenever $-\alpha \notin \sigma\left(\widetilde{\mathbf{M}}(\boldsymbol{\mu})\right)$.

Because $\mathbf{M}_{(1)}(\boldsymbol{\mu})$ may fail to be well defined for some $\alpha$, a more robust approach is to construct preconditioners based on the structure of $\mathbf{A}_h(\boldsymbol{\mu})$. The matrix $\mathbf{P}$ of the RB method has orthonormal columns, so it can be augmented to an orthonormal basis of $\mathbb{R}^{N_h}$. Let
$\mathbf{B} = [\mathbf{P}, \mathbf{Q}]$
be the resulting basis, which is an orthogonal matrix of order $N_h$. Denote
\begin{displaymath}    A_{\mathbf{B}}=\mathbf{B}^{-1}\mathbf{A}_h(\boldsymbol{\mu})\mathbf{B}=\mathbf{B}^T\mathbf{A}_h(\boldsymbol{\mu})\mathbf{B}=\begin{bmatrix}
        \mathbf{A}_{11} & \mathbf{A}_{12} \\
        \mathbf{A}_{21} & \mathbf{A}_{22}
    \end{bmatrix},
\end{displaymath}
where
\begin{displaymath}
\mathbf{A}_{11}=\mathbf{P}^T\mathbf{A}_h(\boldsymbol{\mu})\mathbf{P}, \mathbf{A}_{12}=\mathbf{P}^T\mathbf{A}_h(\boldsymbol{\mu})\mathbf{Q}, \mathbf{A}_{21}=\mathbf{Q}^T\mathbf{A}_h(\boldsymbol{\mu})\mathbf{P}, \mathbf{A}_{22}=\mathbf{Q}^T\mathbf{A}_h(\boldsymbol{\mu})\mathbf{Q}.
\end{displaymath}
If $\mathbf{A}_{11}$ is nonsingular (i.e., the RB reduced system is well posed), its Schur complement is
\begin{displaymath}
    \mathbf{S}=\mathbf{A}_{22}-\mathbf{A}_{21}\mathbf{A}_{11}^{-1}\mathbf{A}_{12}.
\end{displaymath}
Moreover, $\det(A_{\mathbf{B}})=\det(\mathbf{A}_{11})\det(\mathbf{S})$, which implies that $\mathbf{S}$ is nonsingular. Applying elementary row or column operations to the block matrix gives
\begin{displaymath}
    A_{\mathbf{B}}^{-1}=\begin{bmatrix}
\mathbf{A}_{11}^{-1}+\mathbf{A}_{11}^{-1}\mathbf{A}_{12}\mathbf{S}^{-1}\mathbf{A}_{21}\mathbf{A}_{11}^{-1} & -\mathbf{A}_{11}^{-1}\mathbf{A}_{12}\mathbf{S}^{-1}\\ -\mathbf{S}^{-1}\mathbf{A}_{21}\mathbf{A}_{11}^{-1} & \mathbf{S}^{-1}
 \end{bmatrix}.
\end{displaymath}
To find a simple approximation of $\mathbf{A}_h^{-1}(\boldsymbol{\mu})$, we drop the off-diagonal blocks of $A_{\mathbf{B}}^{-1}$ and approximate its main diagonal. Specifically, we replace $\mathbf{A}_{11}^{-1}+\mathbf{A}_{11}^{-1}\mathbf{A}_{12}\mathbf{S}^{-1}\mathbf{A}_{21}\mathbf{A}_{11}^{-1}$ with $\mathbf{A}_{11}^{-1}$, and replace $\mathbf{S}^{-1}$ with $\alpha\mathbf{I}$ or $\alpha\mathbf{Q}^T \mathbf{T}^{-1}\mathbf{Q}$, where $\mathbf{T}\in\mathbb{R}^{N_h\times N_h}$ can be a basic preconditioner and $\alpha\in\mathbb{R}$ is a given small scalar, i.e.,
\begin{equation}\label{eq:block-inverse-approximation}
A_{\mathbf{B}}^{-1}\approx\widetilde{A_{\mathbf{B}}}^{-1}:=\begin{bmatrix}
\mathbf{A}_{11}^{-1} & \mathbf{O}\\

\mathbf{O} & \alpha\mathbf{I}
 \end{bmatrix}\text{or}\begin{bmatrix}
\mathbf{A}_{11}^{-1} & \mathbf{O}\\

\mathbf{O} & \alpha\mathbf{Q}^T\mathbf{T}^{-1}\mathbf{Q}
 \end{bmatrix}.
\end{equation}
Consequently, we have
\begin{equation}\label{eq:arb-inverse-approximation}
\begin{aligned}
            \mathbf{A}_h^{-1}(\boldsymbol{\mu})&=\mathbf{B}A_{\mathbf{B}}^{-1}\mathbf{B}^T\approx\mathbf{B}\widetilde{A_{\mathbf{B}}}^{-1}\mathbf{B}^T\\
            &=\widetilde{\mathbf{M}}(\boldsymbol{\mu})+\alpha\mathbf{Q}\mathbf{Q}^T~\text{or}~\widetilde{\mathbf{M}}(\boldsymbol{\mu})+\alpha\mathbf{Q}\mathbf{Q}^T\mathbf{T}^{-1}\mathbf{Q}\mathbf{Q}^T\\
            &=\widetilde{\mathbf{M}}(\boldsymbol{\mu})+\alpha\left(\mathbf{I}-\mathbf{P}\mathbf{P}^T\right)~\text{or}~\widetilde{\mathbf{M}}(\boldsymbol{\mu})+\alpha\left(\mathbf{I}-\mathbf{P}\mathbf{P}^T\right)\mathbf{T}^{-1}\left(\mathbf{I}-\mathbf{P}\mathbf{P}^T\right).
\end{aligned}
\end{equation}
The last identity of \cref{eq:arb-inverse-approximation} holds because $\mathbf{I}=\mathbf{B}\mathbf{B}^T=\mathbf{P}\mathbf{P}^T+\mathbf{Q}\mathbf{Q}^T$. Thus, the Type II and Type III preconditioners are constructed respectively as
\begin{equation}\label{def:p2}
\mathbf{M}_{(2)}(\boldsymbol{\mu})
:=\left(\widetilde{\mathbf{M}}(\boldsymbol{\mu})+\alpha(\mathbf{I}-\mathbf{P}\mathbf{P}^T)\right)^{-1},
\end{equation}
and
\begin{equation}\label{def:p3}
\mathbf{M}_{(3)}(\boldsymbol{\mu})
:= \left(\widetilde{\mathbf{M}}(\boldsymbol{\mu})+\alpha\left(\mathbf{I}-\mathbf{P}\mathbf{P}^T\right)\mathbf{T}^{-1}(\mathbf{I}-\mathbf{P}\mathbf{P}^T)\right)^{-1}.
\end{equation}

\begin{remark}
Usually, $\widetilde{\mathbf{M}}(\boldsymbol{\mu})$ and $\mathbf{P}\mathbf{P}^T$ are dense matrices of order $N_h$, so directly computing and storing them is impossible. However, when applying a preconditioner $\mathbf{M}$, one does not need to obtain and store $\mathbf{M}$ or $\mathbf{M}^{-1}$ explicitly; it is sufficient to compute $\mathbf{M}^{-1}\mathbf{v}$ for any input vector $\mathbf{v}$. Thus, we need only store $\mathbf{P}$ and $\mathbf{A}_N^{-1}(\boldsymbol{\mu})=\left(\mathbf{P}^T\mathbf{A}_h(\boldsymbol{\mu})\mathbf{P}\right)^{-1}$, and apply multiplication of these matrices (or their transposes) to the input vector in the order specified by the preconditioner definitions. This is easy to implement because we have explicit expressions for $\mathbf{M}^{-1}$ in \cref{def:p1}, \cref{def:p2}, and \cref{def:p3}. It is also inexpensive because $\mathbf{P}$ has only $N$ columns and $\mathbf{A}_N^{-1}(\boldsymbol{\mu})$ is a matrix of order $N$.
\end{remark}

\subsection{Nonsingularity of the resulting preconditioners}
As \cref{prop:shifted-nonsingularity} shows, the Type I preconditioner is not guaranteed to be well defined. However, as long as $-\alpha \notin \sigma\left(\widetilde{\mathbf{M}}(\boldsymbol{\mu})\right)$, $\widetilde{\mathbf{M}}(\boldsymbol{\mu})+\alpha\mathbf{I}$ is nonsingular. If $\alpha$ is selected randomly from a continuous range, the probability that $\widetilde{\mathbf{M}}(\boldsymbol{\mu})+\alpha\mathbf{I}$ is singular is zero. We therefore consider the Type I preconditioner as a possible candidate for a simple low-cost method, with its admissibility checked numerically.

To analyze the nonsingularity of Type II and Type III preconditioners, we first introduce the following lemma.
\begin{lemma}\label{lem:positive-definite-compression}
    Let $\mathbf{A}\in\mathbb{R}^{n\times n}$ be a positive definite matrix, and let $\mathbf{B}\in\mathbb{R}^{n\times m}~(n\geq m)$ be a matrix with full column rank. Then $\mathbf{B}^T\mathbf{A}\mathbf{B}$ is positive definite.
\end{lemma}
\begin{proof}
For any $\mathbf{0}\neq\mathbf{x}\in\mathbb{R}^m$, $\mathbf{B}\mathbf{x}\neq\mathbf{0}$, we have
\begin{displaymath}
\mathbf{x}^T\mathbf{B}^T\mathbf{A}\mathbf{B}\mathbf{x}=\left( \mathbf{B}\mathbf{x} \right)^T\mathbf{A} \left(\mathbf{B}\mathbf{x}\right)>0,
\end{displaymath}
because $\mathbf{A}$ is positive definite.
\end{proof}
We now analyze the nonsingularity of Type II and Type III preconditioners.
\begin{theorem}\label{thm:type-ii-iii-nonsingularity}
    If $\mathbf{A}_{11}$ is nonsingular and $\mathbf{T}$ is definite (either positive definite or negative definite), then for any $\alpha\neq 0$, Type II and Type III preconditioners are both nonsingular.
\end{theorem}
\begin{proof}
    From \cref{eq:block-inverse-approximation} and \cref{eq:arb-inverse-approximation}, we have
    \begin{displaymath}
\begin{aligned}
    \det\left(\widetilde{\mathbf{M}}(\boldsymbol{\mu})+\alpha\mathbf{Q}\mathbf{Q}^T\right)&=\det\left(\mathbf{B}\widetilde{A_{\mathbf{B}}}^{-1}\mathbf{B}^T\right)=\det\left(\widetilde{A_{\mathbf{B}}}^{-1}\right)\\
    &=\alpha^{N_h-N}\det\left(\mathbf{A}_{11}^{-1}\right)\neq 0,
\end{aligned}
    \end{displaymath}
so Type II is nonsingular. Similarly, for Type III, we have
\begin{displaymath}
    \det\left(\widetilde{\mathbf{M}}(\boldsymbol{\mu})+\alpha\mathbf{Q}\mathbf{Q}^T\mathbf{T}^{-1}\mathbf{Q}\mathbf{Q}^T\right)=\alpha^{N_h-N}\det\left(\mathbf{A}_{11}^{-1}\right)\det\left( \mathbf{Q}^T\mathbf{T}^{-1}\mathbf{Q} \right).
\end{displaymath}
As $\mathbf{T}$ is definite, it follows that $\mathbf{T}^{-1}$ is definite. If $\mathbf{T}^{-1}$ is positive definite, then \cref{lem:positive-definite-compression} implies that $\mathbf{Q}^T\mathbf{T}^{-1}\mathbf{Q}$ is positive definite. If $\mathbf{T}^{-1}$ is negative definite, the same argument applied to $-\mathbf{T}^{-1}$ shows that $\mathbf{Q}^T\mathbf{T}^{-1}\mathbf{Q}$ is negative definite. In both cases, $\det\left( \mathbf{Q}^T\mathbf{T}^{-1}\mathbf{Q} \right)\neq 0$.
\end{proof}
\begin{remark}
The Type III preconditioner requires a choice of matrix or operator $\mathbf{T}$.
For example, when the diagonal entries of $\mathbf{A}_h(\boldsymbol{\mu})$ are all nonzero and have the same sign, one may take $\mathbf{T}=\mathbf{D}(\boldsymbol{\mu})$, where $\mathbf{D}(\boldsymbol{\mu})=\operatorname{diag}(\mathbf{A}_h(\boldsymbol{\mu}))$ denotes the diagonal matrix formed from the diagonal entries of $\mathbf{A}_h(\boldsymbol{\mu})$. Type II can also be regarded as Type III with $\mathbf{T}=\mathbf{I}$.
\end{remark}

\subsection{Dynamic construction of ARB preconditioners}\label{sec:dynamic-arb-construction}
Although the ARB preconditioners generally ensure nonsingularity and can be easily applied to any iterative solver, we must still assess whether they are effective. Each ARB preconditioner has two parts: $\widetilde{\mathbf{M}}(\boldsymbol{\mu})$ comes directly from the RB method, while the $\alpha$ term is introduced to ensure nonsingularity. Hence $\alpha$ is usually chosen small, so $\widetilde{\mathbf{M}}(\boldsymbol{\mu})$ dominates the approximation property of the preconditioner. According to the basic theory of the RB method, if the reduced space $V_N$ is well constructed, that is, if the basis $\mathbf{P}$ is well chosen, the RB solution $\mathbf{P}\mathbf{u}_N(\boldsymbol{\mu})=\widetilde{\mathbf{M}}(\boldsymbol{\mu})\mathbf{f}_h(\boldsymbol{\mu})$ can be a good approximation of the full-order solution $\mathbf{u}_h(\boldsymbol{\mu})=\mathbf{A}_h^{-1}(\boldsymbol{\mu})\mathbf{f}_h(\boldsymbol{\mu})$. However, this only means that $\mathbf{u}_h(\boldsymbol{\mu})$ is well approximated in the span of the columns of $\mathbf{P}$; it does not imply that $\widetilde{\mathbf{M}}(\boldsymbol{\mu})$ is a uniformly good approximation of $\mathbf{A}_h^{-1}(\boldsymbol{\mu})$ as an operator. A reduced space trained only on solution snapshots may therefore be insufficient for approximating the corrections encountered later in the iterative process.

At step $k$ of an iterative solution of \eqref{eq:parametrized-linear-system}, the method effectively needs to approximately solve the correction equation
\begin{displaymath}
 \mathbf{A}_h(\boldsymbol{\mu})\mathbf{e}^{(k)}(\boldsymbol{\mu})=\mathbf{r}^{(k)}(\boldsymbol{\mu}),
\end{displaymath}
and update the iterative solution as
\begin{displaymath}
 \mathbf{u}^{(k+1)}(\boldsymbol{\mu})=\mathbf{u}^{(k)}(\boldsymbol{\mu})+\mathbf{e}^{(k)}(\boldsymbol{\mu}),
\end{displaymath}
where $\mathbf{u}^{(k)}(\boldsymbol{\mu})$ is the $k$-th iterative solution and $\mathbf{r}^{(k)}(\boldsymbol{\mu})$ is the corresponding residual at step $k$.
This motivates constructing the reduced space for the RB-type preconditioner $\widetilde{\mathbf{M}}(\boldsymbol{\mu})$ to approximate the correction $\mathbf{e}^{(k)}(\boldsymbol{\mu})$ at step $k$.
Because subsequent corrections may not be adequately represented by the initial reduced space, we use iteration-dependent basis matrices; we denote by $\mathbf{P}_k$ the matrix used at the $k$-th iteration. The construction of $\mathbf{P}_k$ can be carried out similarly using the POD method (\cref{alg:pod}), provided that snapshots of $\mathbf{e}^{(k)}(\boldsymbol{\mu}_i)$ for training parameters $\boldsymbol{\mu}_i \in \Xi_{\text{train}} \subset \mathcal{D}$ are collected during the offline phase. Since such dynamic construction depends on the chosen iterative solver, we discuss the details in \cref{sec:arb-fgmres}, where FGMRES is adopted as the iterative solver.

A similar idea is used in the MSRB preconditioner \cite{dalsanto2017investigation, santo2018multi, dalsanto2019stokes}, but the ARB construction differs. We directly add a simple correction term to ensure nonsingularity, making the method easier to understand and implement.

\subsection{ARB preconditioners for Richardson method}\label{sec:arb-richardson}
In this subsection, we use the Richardson iteration as a simple example to analyze the proposed preconditioners. For simplicity, the Type I form is employed. When the reduced space at step $k$ is spanned by $\mathbf{P}_k$, we write
\begin{displaymath}
    \mathbf{A}_{N,k}(\boldsymbol{\mu})=\mathbf{P}_k^T\mathbf{A}_h(\boldsymbol{\mu})\mathbf{P}_k,\qquad
    \mathbf{M}_k^{-1}(\boldsymbol{\mu})=\mathbf{P}_k\mathbf{A}_{N,k}^{-1}(\boldsymbol{\mu})\mathbf{P}_k^T+\alpha\mathbf{I}.
\end{displaymath}
We set the relaxation parameter to $1$. The ARB-preconditioned Richardson iteration for system \cref{eq:parametrized-linear-system} then becomes
\begin{equation}\label{eq:arb-richardson-iteration}
        \mathbf{u}^{(k+1)}(\boldsymbol{\mu})=\mathbf{u}^{(k)}(\boldsymbol{\mu})+\mathbf{M}_k^{-1}(\boldsymbol{\mu})\mathbf{r}^{(k)}(\boldsymbol{\mu}),\quad k=0, 1, 2, \ldots,
\end{equation}
where $\mathbf{u}^{(k)}(\boldsymbol{\mu})$ is the $k$-th iterative solution and $\mathbf{r}^{(k)}(\boldsymbol{\mu})$ is the corresponding residual at step $k$. For a fixed $\boldsymbol{\mu}$, denote the exact algebraic error by
\begin{displaymath}
    \mathbf{e}^{(k)}(\boldsymbol{\mu}):=\mathbf{u}_h(\boldsymbol{\mu})-\mathbf{u}^{(k)}(\boldsymbol{\mu}),
\end{displaymath}
then $\mathbf{e}^{(k)}(\boldsymbol{\mu})$ satisfies the correction equation
\begin{displaymath}
\mathbf{A}_h(\boldsymbol{\mu})\mathbf{e}^{(k)}(\boldsymbol{\mu})=\mathbf{r}^{(k)}(\boldsymbol{\mu}),\quad k=0, 1, 2, \ldots.
\end{displaymath}
Thus $\mathbf{P}_k$ should be chosen so that $\mathbf{M}_k^{-1}\mathbf{r}^{(k)}$ approximates this correction effectively.

\begin{proposition}\label{prop:richardson-error-bound}
    Fix a vector norm and its induced matrix norm, both denoted by $\|\cdot\|$, and a parameter $\boldsymbol{\mu}\in\mathcal{D}$. Suppose that, for prescribed tolerances $\epsilon_k$, the matrices $\mathbf{P}_{k-1}$ satisfy
    \begin{displaymath}
      \left\| \mathbf{e}^{(k-1)}(\boldsymbol{\mu})- \mathbf{P}_{k-1}\mathbf{A}_{N,k-1}^{-1}(\boldsymbol{\mu})\mathbf{P}_{k-1}^T\mathbf{A}_h(\boldsymbol{\mu})\mathbf{e}^{(k-1)}(\boldsymbol{\mu}) \right\|\leq\epsilon_k \left\| \mathbf{e}^{(k-1)}(\boldsymbol{\mu}) \right\|, ~k=1,2,\ldots.
    \end{displaymath}
Then
\begin{equation}\label{eq:richardson-error-bound}
            \left\|\mathbf{e}^{(k)}(\boldsymbol{\mu})\right\|\leq \left( \prod_{s=1}^k\left(\epsilon_s+C \right) \right) \left\|  \mathbf{e}^{(0)}(\boldsymbol{\mu})\right\|,\quad k=1,2,\ldots,
\end{equation}
where $C=|\alpha|\|\mathbf{A}_h(\boldsymbol{\mu})\|$.
\end{proposition}
\begin{proof}
    Observe that
\begin{displaymath}
\begin{aligned}
        \mathbf{e}^{(k)}(\boldsymbol{\mu})&=\mathbf{u}_h(\boldsymbol{\mu})-\mathbf{u}^{(k-1)}(\boldsymbol{\mu})+\mathbf{u}^{(k-1)}(\boldsymbol{\mu})-\mathbf{u}^{(k)}(\boldsymbol{\mu})\\
        &=\mathbf{e}^{(k-1)}-\mathbf{M}_{k-1}^{-1}(\boldsymbol{\mu})\mathbf{r}^{(k-1)}(\boldsymbol{\mu})=\left( \mathbf{I}- \mathbf{M}_{k-1}^{-1}(\boldsymbol{\mu})\mathbf{A}_h(\boldsymbol{\mu}) \right)\mathbf{e}^{(k-1)}(\boldsymbol{\mu}).
\end{aligned}
\end{displaymath}
Then
\begin{displaymath}
 \begin{aligned}
        \left\| \mathbf{e}^{(k)}(\boldsymbol{\mu}) \right\|&=\left\| \left( \mathbf{I}- \mathbf{M}_{k-1}^{-1}(\boldsymbol{\mu})\mathbf{A}_h(\boldsymbol{\mu}) \right)\mathbf{e}^{(k-1)}(\boldsymbol{\mu}) \right\|\\
        &=\left\| \mathbf{e}^{(k-1)}(\boldsymbol{\mu})- \mathbf{P}_{k-1}\mathbf{A}_{N,k-1}^{-1}(\boldsymbol{\mu})\mathbf{P}_{k-1}^T\mathbf{A}_h(\boldsymbol{\mu})\mathbf{e}^{(k-1)}(\boldsymbol{\mu}) \right. \\&\quad\left.-\alpha\mathbf{A}_h(\boldsymbol{\mu})\mathbf{e}^{(k-1)}(\boldsymbol{\mu})  \right\|\\
        &\leq \left( \epsilon_k+|\alpha|\left\|\mathbf{A}_h(\boldsymbol{\mu}) \right\| \right) \left\| \mathbf{e}^{(k-1)}(\boldsymbol{\mu}) \right\|.
 \end{aligned}
\end{displaymath}
Proceeding recursively, we obtain \cref{eq:richardson-error-bound}.
\end{proof}
A corresponding residual result also holds.
\begin{proposition}\label{prop:richardson-residual-bound}
    Fix a vector norm and its induced matrix norm, both denoted by $\|\cdot\|$, and a parameter $\boldsymbol{\mu}\in\mathcal{D}$. Suppose that, for prescribed tolerances $\epsilon_k$, the matrices $\mathbf{P}_{k-1}$ satisfy
    \begin{displaymath}
      \left\| \mathbf{r}^{(k-1)}(\boldsymbol{\mu})- \mathbf{A}_h(\boldsymbol{\mu})\mathbf{P}_{k-1}\mathbf{A}_{N,k-1}^{-1}(\boldsymbol{\mu})\mathbf{P}_{k-1}^T\mathbf{r}^{(k-1)}(\boldsymbol{\mu}) \right\|\leq\epsilon_k \left\| \mathbf{r}^{(k-1)}(\boldsymbol{\mu}) \right\|,~ k=1,2,\ldots.
    \end{displaymath}
Then
\begin{equation}\label{eq:richardson-residual-bound}
            \left\|\mathbf{r}^{(k)}(\boldsymbol{\mu})\right\|\leq \left( \prod_{s=1}^k\left(\epsilon_s+C \right) \right) \left\|  \mathbf{r}^{(0)}(\boldsymbol{\mu})\right\|,\quad k=1,2,\ldots,
\end{equation}
where $C=|\alpha|\|\mathbf{A}_h(\boldsymbol{\mu})\|$.
\end{proposition}
\begin{proof}
From \cref{eq:arb-richardson-iteration}, the residual at iteration $k$ is
 \begin{displaymath}
     \mathbf{r}^{(k)}(\boldsymbol{\mu})=\mathbf{f}_h(\boldsymbol{\mu})-\mathbf{A}_h(\boldsymbol{\mu})\mathbf{u}^{(k)}(\boldsymbol{\mu})=\mathbf{r}^{(k-1)}(\boldsymbol{\mu})-\mathbf{A}_h(\boldsymbol{\mu})\mathbf{M}^{-1}_{k-1}(\boldsymbol{\mu})\mathbf{r}^{(k-1)}(\boldsymbol{\mu}).
 \end{displaymath}
 Then
\begin{displaymath}
\begin{aligned}
        \left\| \mathbf{r}^{(k)}(\boldsymbol{\mu}) \right\|&=\left\|  \mathbf{r}^{(k-1)}(\boldsymbol{\mu})-\mathbf{A}_h(\boldsymbol{\mu})\mathbf{M}^{-1}_{k-1}(\boldsymbol{\mu})\mathbf{r}^{(k-1)}(\boldsymbol{\mu}) \right\|\\
        &=\left\|\left( \mathbf{I} - \mathbf{A}_h(\boldsymbol{\mu})\mathbf{P}_{k-1}\mathbf{A}_{N,k-1}^{-1}(\boldsymbol{\mu})\mathbf{P}_{k-1}^T-\alpha\mathbf{A}_h(\boldsymbol{\mu}) \right) \mathbf{r}^{(k-1)}(\boldsymbol{\mu})  \right\|\\
        &\leq \left(\epsilon_k+|\alpha|\|\mathbf{A}_h(\boldsymbol{\mu}) \|  \right)\left\| \mathbf{r}^{(k-1)}(\boldsymbol{\mu}) \right\|.
\end{aligned}
\end{displaymath}
 Proceeding recursively, we obtain \cref{eq:richardson-residual-bound}.
\end{proof}
\begin{remark}
    We emphasize that a greedy construction can enforce the hypotheses of \cref{prop:richardson-error-bound} and \cref{prop:richardson-residual-bound} for the training parameters when the corresponding error or residual tolerances $\epsilon_k$ are prescribed. When POD is used to construct the RB spaces, as in our numerical experiments, the POD tolerance controls an aggregate projection error over the snapshot set. The resulting POD spaces can be assessed numerically through the observed convergence behavior.
\end{remark}

\subsection{ARB preconditioners for FGMRES}\label{sec:arb-fgmres}
In the previous sections, we constructed the ARB preconditioners and analyzed their convergence properties using the Richardson iteration for illustrative purposes. To employ a more efficient Krylov subspace method in practical applications, we now consider FGMRES \cite{saad1993flexible, saad2003iterative}. This choice is particularly well suited to our approach, as the ARB preconditioner changes during the iterations, and FGMRES is specifically designed to handle such iteration-dependent operators.

For simplicity, we adopt the standard FGMRES algorithm as described in \cite{saad1993flexible, saad2003iterative}. In our implementation, we use the Euclidean norm of the relative residual of the iterative solution as the stopping criterion, requiring it to be smaller than a prescribed tolerance $\epsilon_{\text{rtol}}$.
\begin{algorithm}[t]
  \caption{FGMRES (as formulated in \cite{saad2003iterative})}\label{alg:fgmres}
\begin{algorithmic}[1]
    \STATE Compute $\mathbf{r}_0 = \mathbf{b} - \mathbf{A}\mathbf{u}_0$, $\beta = \|\mathbf{r}_0\|_2$, and $\mathbf{v}_1 = \mathbf{r}_0 / \beta$
    \FOR{$k = 1,\dots,m$}
      \STATE  Compute $\mathbf{z}_k = \mathbf{M}_k^{-1}\mathbf{v}_k$
      \STATE  Compute $\mathbf{w} = \mathbf{A}\mathbf{z}_k$
      \FOR{$j = 1,\dots,k$}
        \STATE  $h_{j,k} = (\mathbf{w}, \mathbf{v}_j)$
        \STATE  $\mathbf{w} = \mathbf{w} - h_{j,k}\mathbf{v}_j$
      \ENDFOR
      \STATE  Compute $h_{k+1,k} = \|\mathbf{w}\|_2$ and $\mathbf{v}_{k+1} = \mathbf{w} / h_{k+1,k}$
      \STATE  Define $\mathbf{Z}_m = [\mathbf{z}_1,\ldots,\mathbf{z}_m]$, $\overline{\mathbf{H}}_m = \{h_{j,k}\}_{1 \le j \le k+1; 1 \le k \le m}$
    \ENDFOR
    \STATE  Compute $\mathbf{y}_m = \text{argmin}_{\mathbf{y}\in\mathbb{R}^m}\left\|\beta\mathbf{e}_1 - \overline{\mathbf{H}}_m\mathbf{y}\right\|_2$ and $\mathbf{u}_m = \mathbf{u}_0 + \mathbf{Z}_m\mathbf{y}_m$
    \STATE  If satisfied Stop, else set $\mathbf{u}_0 \leftarrow \mathbf{u}_m$ and GoTo 1
\end{algorithmic}
\end{algorithm}

In \cref{alg:fgmres}, $\mathbf{M}_k$ is the preconditioner at iteration $k$. Since its inverse is applied to $\mathbf{v}_k$, $\mathbf{P}_k$ should be constructed for the problem $\mathbf{A}\mathbf{x}_k=\mathbf{v}_k$. In the ARB case, the $k$-th RB space should therefore be trained to solve
\begin{displaymath}
\mathbf{A}_h(\boldsymbol{\mu})\mathbf{x}_k(\boldsymbol{\mu})=\mathbf{v}_k(\boldsymbol{\mu}).
\end{displaymath}
For $k=1$, we have
\begin{displaymath}
\mathbf{x}_1(\boldsymbol{\mu})=\mathbf{A}_h^{-1}(\boldsymbol{\mu})\mathbf{v}_1(\boldsymbol{\mu})=\frac{1}{\beta(\boldsymbol{\mu})}\mathbf{A}_h^{-1}(\boldsymbol{\mu})\mathbf{r}_0(\boldsymbol{\mu})=\frac{1}{\beta(\boldsymbol{\mu})}(\mathbf{u}_h(\boldsymbol{\mu}) -\mathbf{u}_0).
\end{displaymath}
From \cref{alg:fgmres}, we have
\begin{displaymath}
    \mathbf{v}_{k+1}(\boldsymbol{\mu})=\frac{1}{h_{k+1, k}(\boldsymbol{\mu})}\Bigg( \mathbf{A}_h(\boldsymbol{\mu})\mathbf{M}_k^{-1}(\boldsymbol{\mu})\mathbf{v}_k(\boldsymbol{\mu})-\sum_{j=1}^kh_{j,k}(\boldsymbol{\mu})\mathbf{v}_j(\boldsymbol{\mu}) \Bigg),\quad k=1,2,\ldots,
\end{displaymath}
Therefore, we have
\begin{displaymath}
\begin{aligned}
\mathbf{x}_{k+1}(\boldsymbol{\mu})&=\mathbf{A}_h^{-1}(\boldsymbol{\mu})\frac{1}{h_{k+1,k}(\boldsymbol{\mu})}\Bigg( \mathbf{A}_h(\boldsymbol{\mu})\mathbf{M}_k^{-1}(\boldsymbol{\mu})\mathbf{v}_k(\boldsymbol{\mu})-\sum_{j=1}^kh_{j,k}(\boldsymbol{\mu})\mathbf{v}_j(\boldsymbol{\mu}) \Bigg)\\
&=\frac{1}{h_{k+1,k}(\boldsymbol{\mu})}\Bigg(\mathbf{M}_k^{-1}(\boldsymbol{\mu})\mathbf{v}_k(\boldsymbol{\mu})-\sum_{j=1}^kh_{j,k}(\boldsymbol{\mu})\mathbf{A}_h^{-1}(\boldsymbol{\mu})\mathbf{v}_j(\boldsymbol{\mu}) \Bigg)\\
&=\frac{1}{h_{k+1,k}(\boldsymbol{\mu})}\Bigg( \mathbf{z}_k(\boldsymbol{\mu})-\sum_{j=1}^kh_{j,k}(\boldsymbol{\mu})\mathbf{x}_j(\boldsymbol{\mu}) \Bigg),\quad k=1,2,\ldots.
\end{aligned}
\end{displaymath}
This gives the recursive formula
\begin{equation}\label{eq:fgmres-snapshot-recursion}
\begin{cases}
\beta(\boldsymbol{\mu})=\|\mathbf{f}_h(\boldsymbol{\mu})-\mathbf{A}_h(\boldsymbol{\mu})\mathbf{u}_0 \|_2,\\
\mathbf{x}_1(\boldsymbol{\mu})=\frac{1}{\beta(\boldsymbol{\mu}) }  \left(\mathbf{u}_h(\boldsymbol{\mu})-\mathbf{u}_0\right),\\
\mathbf{x}_{k+1}(\boldsymbol{\mu})=\frac{1}{h_{k+1,k}(\boldsymbol{\mu})}\left( \mathbf{z}_k(\boldsymbol{\mu})-\sum_{j=1}^kh_{j,k}(\boldsymbol{\mu})\mathbf{x}_j(\boldsymbol{\mu}) \right),\quad k\geq 1.
\end{cases}
\end{equation}
For each training parameter, \cref{eq:fgmres-snapshot-recursion} requires a large-scale solve only in the first step to obtain the initial snapshot; the remaining snapshots are obtained by subsequent recursive steps. This keeps the offline computation time acceptable in practice. 

We now give the complete offline algorithm in \cref{alg:offline}, which describes the dynamic construction of RB spaces during the offline phase. In \cref{alg:offline}, the function POD($\cdot$) in lines 6 and 18 is the one described in \cref{alg:pod}. The solver used in line 4 may be a direct solver or another iterative solver with the desired accuracy. Although lines 2--5 and 8--14 of the algorithm are presented in serial form, in practice one can solve the training problems in parallel, greatly reducing the offline time. Line 9 uses the Type I preconditioner, which can be easily replaced by either the Type II or Type III preconditioner.

\begin{algorithm}[t]
  \caption{Offline Dynamic Generation of RB Spaces}\label{alg:offline}
  \begin{algorithmic}[1]
    \STATE \textbf{procedure} DGRBS$\left(\{\boldsymbol{\mu}_t\}_{t=1}^{n_s}, \mathbf{u}_0, L, \epsilon_{\text{tol}}, \alpha\right)$
\FOR{$t=1, 2, \ldots, n_s$}
\STATE Compute $\mathbf{r}_0(\boldsymbol{\mu}_t) = \mathbf{f}_h(\boldsymbol{\mu}_t) - \mathbf{A}_h(\boldsymbol{\mu}_t)\mathbf{u}_0$, $\beta(\boldsymbol{\mu}_t) = \|\mathbf{r}_0(\boldsymbol{\mu}_t)\|_2$, and $\mathbf{v}_1(\boldsymbol{\mu}_t) = \mathbf{r}_0(\boldsymbol{\mu}_t) / \beta(\boldsymbol{\mu}_t)$
\STATE Compute $\mathbf{u}_h(\boldsymbol{\mu}_t)=\mathbf{A}_h^{-1}(\boldsymbol{\mu}_t)\mathbf{f}_h(\boldsymbol{\mu}_t)$, $\mathbf{x}_1(\boldsymbol{\mu}_t)=(\mathbf{u}_h(\boldsymbol{\mu}_t)-\mathbf{u}_0)/\beta(\boldsymbol{\mu}_t)$
\ENDFOR
\STATE Compute $\mathbf{P}_1=\mathrm{POD}\left(\mathbf{x}_1(\boldsymbol{\mu}_1), \mathbf{x}_1(\boldsymbol{\mu}_2), \ldots, \mathbf{x}_1(\boldsymbol{\mu}_{n_s}); \epsilon_{\text{tol}}\right)$

\FOR{$k=1, 2, \ldots, L-1$}
    \FOR{$t = 1, 2, \dots, n_s$}
      \STATE  Compute $\mathbf{A}_{N,k}(\boldsymbol{\mu}_t)=\mathbf{P}_k^T\mathbf{A}_h(\boldsymbol{\mu}_t)\mathbf{P}_k$ and $\mathbf{z}_k(\boldsymbol{\mu}_t) = \mathbf{P}_k\mathbf{A}_{N,k}^{-1}(\boldsymbol{\mu}_t)\mathbf{P}_k^T\mathbf{v}_k(\boldsymbol{\mu}_t)+\alpha\mathbf{v}_k(\boldsymbol{\mu}_t)$
      \STATE  Compute $\mathbf{w}(\boldsymbol{\mu}_t) = \mathbf{A}_h(\boldsymbol{\mu}_t)\mathbf{z}_k(\boldsymbol{\mu}_t)$
      \FOR{$j = 1,\dots,k$}
        \STATE $h_{j,k}(\boldsymbol{\mu}_t) = (\mathbf{w}(\boldsymbol{\mu}_t), \mathbf{v}_j(\boldsymbol{\mu}_t))$
        \STATE $\mathbf{w}(\boldsymbol{\mu}_t) = \mathbf{w}(\boldsymbol{\mu}_t) - h_{j,k}(\boldsymbol{\mu}_t) \mathbf{v}_j(\boldsymbol{\mu}_t)$
      \ENDFOR
      \STATE Compute $h_{k+1,k}(\boldsymbol{\mu}_t) = \|\mathbf{w}(\boldsymbol{\mu}_t)\|_2$ and $\mathbf{v}_{k+1}(\boldsymbol{\mu}_t) = \mathbf{w}(\boldsymbol{\mu}_t) / h_{k+1,k}(\boldsymbol{\mu}_t)$
      \STATE Compute $\mathbf{x}_{k+1}(\boldsymbol{\mu}_t)=\frac{1}{h_{k+1,k}(\boldsymbol{\mu}_t)}\left( \mathbf{z}_k(\boldsymbol{\mu}_t)-\sum_{j=1}^kh_{j,k}(\boldsymbol{\mu}_t)\mathbf{x}_j(\boldsymbol{\mu}_t) \right)$
    \ENDFOR
\STATE  Compute $\mathbf{P}_{k+1}=\text{POD}\left(\mathbf{x}_{k+1}(\boldsymbol{\mu}_1),\ldots,\mathbf{x}_{k+1}(\boldsymbol{\mu}_{n_s}); \epsilon_{\text{tol}}\right)$
\ENDFOR

    \STATE Save $\mathbf{P}_1, \mathbf{P}_2,\ldots, \mathbf{P}_{L}$
     \STATE \textbf{end procedure}
  \end{algorithmic}
\end{algorithm}

\begin{remark}
   In \cref{alg:offline}, we construct only $L$ reduced spaces $\mathbf{P}_k$, $k=1, 2,\ldots, L$. Substituting them into the preconditioner definitions yields the corresponding ARB preconditioner used by FGMRES at the $k$-th iteration. If the number of FGMRES iterations exceeds $L$, we fix $\mathbf{P}_{L}$ in the preconditioners for the remaining iterations.
\end{remark}

\section{Numerical experiments}\label{sec:numerical-experiments}
In this section, we compare the proposed ARB-preconditioned FGMRES method with standard GMRES equipped with two classical algebraic preconditioners: ILU($k$) and BoomerAMG. Such algebraic and multilevel preconditioners serve as standard baselines for large sparse PDE systems \cite{saad2003iterative, wathen2015preconditioning}. The POD tolerance $\epsilon_{\text{tol}}$ used in the offline phase is set to $10^{-3}$, and the convergence criterion for the linear solver is set to a relative residual norm of $\epsilon_{\text{rtol}} = 1.0 \times 10^{-7}$. The analysis in \cref{sec:dynamic-arb-construction} and \cref{prop:shifted-nonsingularity} suggests that a small value of $\alpha$ is preferred. Balancing stability and accuracy, we select $\alpha = 10^{-4}$. We perform tests on four benchmark problems: the steady-state convection-diffusion equation, the anisotropic vortex problem, the steady-state Stokes equation, and the Helmholtz equation. For each problem, the fill-in level $k$ of the ILU($k$) preconditioner is chosen as the largest level for which the solve time still improves significantly upon ILU($k-1$); a larger $k$ is not used, since the additional fill-in raises the memory cost while the speed gain becomes marginal. We use the default BoomerAMG settings without additional parameter tuning. For the Stokes problem, BoomerAMG is applied within the field-split preconditioner described in Case 3. Unless otherwise stated, the $n_s$ offline training parameters are generated using a uniform sampling strategy over the prescribed parameter domains. In the online stage, each reported statistic is obtained by solving the same problem for $100$ randomly generated parameters that differ from those used in the offline stage. All numerical experiments are conducted on a personal computer equipped with an i9-14900HX CPU.

\subsection{Case 1: The convection-diffusion equation}
We consider the steady-state convection-diffusion equation, a classical problem studied extensively in \cite{ernst2000residual, stynes2005steady}. 
The governing equation is
\begin{displaymath}
    \begin{cases}
    -\mu\Delta u + \mathbf{b}\cdot\nabla u = f, & \text{in}~\Omega,\\
    u = 0, & \text{on}~\partial\Omega,
    \end{cases}
\end{displaymath}
where $u: \Omega \to \mathbb{R}$ is the scalar field, $\Omega \subset \mathbb{R}^2$ is the domain, $\mathbf{b} \in \mathbb{R}^2$ is a constant convection vector, and $f$ is a given source function. The scalar $\mu > 0$ represents the diffusion coefficient and serves as the parameter characterizing the problem. In the following experiments, we set $\mathbf{b} = [1, 2]^T$, $f \equiv 1$, and define the domain as $\Omega = (0, 1) \times (0, 1)$.

The discrete problem is formulated using continuous piecewise linear Lagrange finite elements ($\mathcal{P}_1$) on a uniform triangular mesh $\mathcal{T}_h$ generated from a $700\times 700$ grid. The number of degrees of freedom is $N_h = 491,401$. However, in the convection-dominated regime (i.e., for small $\mu$), the standard Galerkin method is known to suffer from spurious oscillations. To mitigate this, we employ the Streamline Upwind Petrov-Galerkin (SUPG) stabilization technique \cite{brooks1982streamline, franca2006revisiting}. For $\mathcal{P}_1$ elements, the stabilized discrete weak form is
\begin{displaymath}
    \int_\Omega \mu\nabla u\cdot\nabla v~\mathrm{d}x + \int_\Omega (\mathbf{b}\cdot\nabla u)v~\mathrm{d}x + \sum_{K\in\mathcal{T}_h} \delta_K \int_K (\mathbf{b}\cdot\nabla u)(\mathbf{b}\cdot\nabla v)~\mathrm{d}x = \int_\Omega fv~\mathrm{d}x,
\end{displaymath}
 The stabilization parameter is defined by
\begin{displaymath}
    \delta_K = \frac{h_K}{2\|\mathbf{b}\|_2} \left( \coth(\mathrm{Pe}_K) - \frac{1}{\mathrm{Pe}_K} \right), \quad \text{with} \quad \mathrm{Pe}_K := \frac{\|\mathbf{b}\|_2 h_K}{2\mu}.
\end{displaymath}
The global mesh Péclet number, used to determine whether stabilization is required, is $\mathrm{Pe}(\mu; \mathbf{b}, h) := \frac{\|\mathbf{b}\|_2h}{2\mu}$. For the given problem configuration, the critical threshold corresponds to
\begin{displaymath}
    \mathrm{Pe}(\mu; \mathbf{b}, h) = \frac{\sqrt{5}\sqrt{2}}{1400\mu} = 1 \quad \Longrightarrow \quad \mu = \frac{\sqrt{10}}{1400} \approx 2.25 \times 10^{-3}.
\end{displaymath}
Consequently, the SUPG method is activated when $\mu \leq 2.25 \times 10^{-3}$.

\begin{remark}
    For the convection-diffusion equation, the coefficient matrix $\mathbf{A}_{h}(\boldsymbol{\mu})$ arising from the finite element discretization typically has positive diagonal entries.
Consequently, the diagonal matrix $\mathbf{D}(\boldsymbol{\mu}) = \mathrm{diag}(\mathbf{A}_{h}(\boldsymbol{\mu}))$ is positive definite, which permits the use of the Type III ARB preconditioner with $\mathbf{T} = \mathbf{D}(\boldsymbol{\mu})$.
  \end{remark}

\newcolumntype{R}{>{$}r<{$}}
\newcolumntype{V}[1]{>{[\;}*{#1}{R@{\;\;}}R<{\;]}}
\begin{table}[t]
 \footnotesize
 \captionsetup{position=top}
 \caption{Numerical results for $100$ random online solves of the convection-diffusion equation. Here and throughout, results reported with $\pm$ are given as mean $\pm$ standard deviation.}\label{tab:convection-diffusion}
 \begin{center}
 \subfloat[$\mu\in [0.1, 1\text{]}$: $n_s=40, L=6$. ]{\label{tab:convection-diffusion-moderate}
 \scalebox{0.97}{
\begin{tabular}{|c|c|c|c|c|c|}
\hline
KSP Type
& \multicolumn{2}{c|}{GMRES}
& \multicolumn{3}{c|}{FGMRES}\\
\hline

PC Type
& ILU(2)
& AMG
& $\mathbf{M}_{(1)}$
& $\mathbf{M}_{(2)}$
& $\mathbf{M}_{(3)}$\\
\hline

Iterations
& $357.95\pm53.24$
& $5.05\pm0.22$
& $3.64\pm0.69$
& $3.64\pm0.69$
& $3.49\pm0.56$\\

\hline

$t_{\mathrm{online}}(s)$
& $4.391\pm0.644$
& $0.525\pm0.019$
& $0.159\pm0.023$
& $0.167\pm0.025$
& $0.170\pm0.020$\\
\hline

$t_{\mathrm{offline}}(s)$
& --
& --
& $104.27$
& $106.51$
& $107.12$
\\
\hline
\end{tabular}
 }
 }

 \subfloat[$\mu\in[10^{-5}, 2.25\times 10^{-3}\text{]}$: $n_s=50, L=6$.]{\label{tab:convection-diffusion-dominated}
	\scalebox{0.97}{
\begin{tabular}{|c|c|c|c|c|c|}
\hline
KSP Type
& \multicolumn{2}{c|}{GMRES}
& \multicolumn{3}{c|}{FGMRES}\\
\hline

PC Type
& ILU(2)
& AMG
& $\mathbf{M}_{(1)}$
& $\mathbf{M}_{(2)}$
& $\mathbf{M}_{(3)}$\\
\hline

Iterations
& $6.13\pm 2.31$
& $16.64\pm 4.69$
& $4.06\pm 0.40$
& $4.06\pm 0.40$
& $4.04\pm0.40$\\

\hline

$t_{\mathrm{online}}(\mathrm{s})$
& $0.234\pm 0.034$
& $1.622\pm 0.256$
& $0.217\pm 0.030$
& $0.223\pm 0.031$
& $0.245\pm 0.036$\\

\hline

$t_{\mathrm{offline}}(\mathrm{s})$
& --
& --
& $147.75$
& $149.29$
& $152.05$
\\
\hline
\end{tabular}
 }
 }
 \end{center}
 \end{table}

\begin{table}[t]
\scriptsize
\caption{Speed-up factors for the convection-diffusion equation over the $100$ online solves. Each speed-up is computed as $t_{\mathrm{ref}}/t_{\mathrm{ARB}}$ for the same parameter sample.}\label{tab:convection-diffusion-speedup}
\begin{center}
\setlength{\tabcolsep}{3pt}
\renewcommand{\arraystretch}{1.05}
\resizebox{\textwidth}{!}{
\begin{tabular}{|c|c|c|c|c|c|c|c|c|c|c|}
\hline
{}
& {}
& \multicolumn{3}{c|}{$\mathbf{M}_{(1)}$}
& \multicolumn{3}{c|}{$\mathbf{M}_{(2)}$}
& \multicolumn{3}{c|}{$\mathbf{M}_{(3)}$}\\
\hline
Parameter range
& Reference solver
& Max
& Mean
& Min
& Max
& Mean
& Min
& Max
& Mean
& Min\\
\hline
$\mu\in[0.1,1]$
& GMRES+ILU(2)
& $33.0\times$
& $28.1\times$
& $11.5\times$
& $31.8\times$
& $26.7\times$
& $10.7\times$
& $30.9\times$
& $26.1\times$
& $14.6\times$\\
\hline
$\mu\in[0.1,1]$
& GMRES+AMG
& $4.1\times$
& $3.3\times$
& $1.9\times$
& $3.8\times$
& $3.1\times$
& $1.8\times$
& $3.8\times$
& $3.1\times$
& $2.3\times$\\
\hline
$\mu\in[10^{-5},2.25\times10^{-3}]$
& GMRES+ILU(2)
& $1.3\times$
& $1.1\times$
& $0.9\times$
& $1.3\times$
& $1.1\times$
& $0.9\times$
& $1.2\times$
& $1.0\times$
& $0.8\times$\\
\hline
$\mu\in[10^{-5},2.25\times10^{-3}]$
& GMRES+AMG
& $11.6\times$
& $7.7\times$
& $4.6\times$
& $11.5\times$
& $7.5\times$
& $4.4\times$
& $10.9\times$
& $6.8\times$
& $3.4\times$\\
\hline
\end{tabular}
}
\end{center}
\end{table}

\begin{figure}[t]
\footnotesize
 \centering
 \subfloat[$\mu=1.793571\times10^{-3}$ (the worst case of $\mu\in[10^{-5}, 2.25\times10^{-3}\text{]}$).]{\label{fig:convection-diffusion-dominated-residual}
  \includegraphics[width=0.45\textwidth]{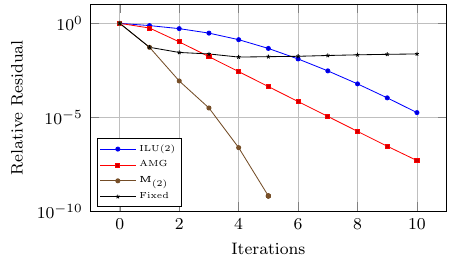}
 }
 \hspace{0.01\textwidth}
 \subfloat[$\mu=0.106088$ (the worst case of $\mu\in[0.1,1\text{]}$).]{\label{fig:convection-diffusion-moderate-residual}
  \includegraphics[width=0.45\textwidth]{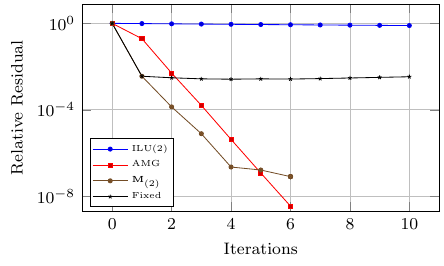}
 }
 \caption{Relative residual norms with different preconditioners for the convection-diffusion equation.}
 \label{fig:convection-diffusion-residuals}
\end{figure}

The results in \cref{tab:convection-diffusion-moderate} correspond to the non-convection-dominated range $\mu \in [0.1, 1]$. In the offline stage, we use $n_s=40$ uniformly sampled training parameters and construct $L=6$ dynamic reduced spaces. In the $100$ online tests, ILU(2) requires $357.95\pm53.24$ iterations on average and $4.391\pm0.644$ seconds per solve. AMG requires $5.05\pm0.22$ iterations and $0.525\pm0.019$ seconds per solve, whereas the ARB preconditioners require only about $3.5$--$3.6$ iterations and $0.16$--$0.17$ seconds. Thus, the proposed preconditioners retain a pronounced online advantage over both ILU(2) and AMG.

The results in \cref{tab:convection-diffusion-dominated} focus exclusively on the convection-dominated, SUPG-stabilized range $\mu \in [10^{-5}, 2.25\times 10^{-3}]$. In the offline stage, we construct $L=6$ dynamic reduced spaces from $n_s=50$ logarithmically sampled parameters; equivalently, the training values form a geometric progression over the interval. In the online stage, the reported results are obtained from $100$ independently and log-uniformly randomly sampled parameters in the same interval, so that each order of magnitude of $\mu$ receives a comparable number of samples. ILU(2) requires $6.13\pm2.31$ iterations and $0.234\pm0.034$ seconds on average, while AMG requires $16.64\pm4.69$ iterations and $1.622\pm0.256$ seconds. The ARB preconditioners reduce the iteration count to about $4.0$, and their online times, $0.217$--$0.245$ seconds, are comparable with that of ILU(2) and substantially faster than AMG.

The speed-up statistics in \cref{tab:convection-diffusion-speedup} make this contrast explicit. In the non-convection-dominated range, the mean speed-ups of the ARB methods are $26.1\times$--$28.1\times$ over ILU(2) and $3.1\times$--$3.3\times$ over AMG. In the SUPG-stabilized convection-dominated range, the mean speed-ups relative to ILU(2) are only $1.0\times$--$1.1\times$, and some samples are slower than ILU(2); relative to AMG, the mean speed-ups are $6.8\times$--$7.7\times$. Note that for the SUPG-stabilized convection-dominated case, ILU(2) is extremely efficient, about $7\times$ faster than AMG. Even in this extreme case, our proposed ARB preconditioners are as efficient as ILU, and for all other tested cases, ARB preconditioners become much more efficient than ILU.

For completeness, we report a break-even estimate for the comparison of Type II preconditioner to AMG. For $\mu\in[0.1,1]$, the offline time is $t_{\mathrm{offline}} = 106.51$ seconds and the online solve time reduction is approximately $0.358$ seconds per solve compared with AMG ($0.167\,\mathrm{s}$ vs $0.525\,\mathrm{s}$). For $\mu \in [10^{-5}, 2.25\times 10^{-3}]$, the offline time is $149.29$ seconds and the online solve time reduction is approximately $1.399$ seconds per solve ($0.223\,\mathrm{s}$ vs $1.622\,\mathrm{s}$). The corresponding break-even points are estimated as
    \begin{displaymath}
        N_{\text{break-even}} = \frac{106.51}{0.358} \approx 298 \text{ solves}, \qquad
        N_{\text{break-even}} = \frac{149.29}{1.399} \approx 107 \text{ solves}.
    \end{displaymath}
    This implies that the Type II ARB preconditioner becomes cost-effective after about $298$ parametrized solves in the moderate-diffusion range and about $107$ solves in the convection-dominated range.

For the 100 random online tests in each parameter range, we select the worst case for the Type II preconditioner and plot the corresponding residual histories in \cref{fig:convection-diffusion-residuals}. Even in these worst cases, the proposed ARB preconditioner converges as rapidly as AMG. In \cref{fig:convection-diffusion-residuals}, the curve labeled Fixed is a control test in which the first ARB preconditioner, namely the one associated with the first reduced space $\mathbf{P}_1$, is kept fixed throughout the FGMRES iteration instead of being dynamically updated. The first steps of the dynamically constructed ARB and Fixed curves are identical. This initial decrease reflects the effectiveness of the first ARB-preconditioned search direction. The first FGMRES step includes both the additive correction and residual minimization, so its accuracy generally differs from that of a standalone Galerkin RB solve. With this relatively small training set, further iterations are needed to reach the prescribed tolerance, especially in the convection-dominated regime. Using the dynamically constructed ARB preconditioners, the remaining Krylov iterations reduce the residual to the prescribed tolerance in only a few steps. The Fixed curve shows a clear reduction in relative residual norm at the first step, but its subsequent convergence is much less effective. This behavior is consistent with the discussion in \cref{sec:dynamic-arb-construction}: the first reduced space can approximate the initial correction well, whereas subsequent corrections may not be adequately represented by that space, motivating the use of iteration-dependent reduced spaces.

\subsection{Case 2: The anisotropic vortex problem}
\begin{figure}[ht]
\footnotesize
 \centering
 \subfloat[$\boldsymbol{\mu}=(19.665701, 18.256891)$ (the worst case).]{\label{fig:anisotropic-vortex-worst-residual}
  \includegraphics[width=0.45\textwidth]{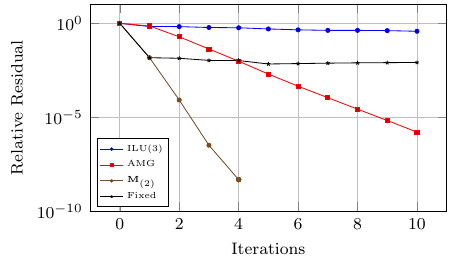}
 }
 \hspace{0.01\textwidth}
 \subfloat[$\boldsymbol{\mu}=(17.401796, 17.512798)$ (the best case).]{\label{fig:anisotropic-vortex-best-residual}
  \includegraphics[width=0.45\textwidth]{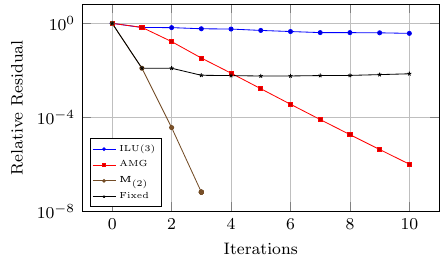}
 }
 \caption{Relative residual norms with different preconditioners for the anisotropic vortex problem.}
 \label{fig:anisotropic-vortex-residuals}
\end{figure}
In this subsection, we consider a convection-diffusion-reaction equation, a variant of the convection-diffusion equation motivated by transport-dominated benchmark problems \cite{dalsanto2017investigation, stynes2005steady}, on the unit square domain $\Omega = (0, 1)^2$:
\begin{displaymath}
    \begin{cases}
    -\varepsilon\Delta u + \mathbf{b}(\mathbf{x}; \boldsymbol{\mu})\cdot\nabla u + \sigma u = f, & \text{in}~\Omega,\\
    u = 0, & \text{on}~\partial\Omega,
    \end{cases}
\end{displaymath}
where $\varepsilon = 0.005$ is the constant diffusion coefficient and $\sigma = 1$ is the reaction coefficient. The source term $f$ is a localized Gaussian function centered at $(0.75, 0.5)$, given by
\begin{displaymath}
    f(x, y) = \exp\left(-100\left((x - 0.75)^2 + (y - 0.5)^2\right)\right).
\end{displaymath}
This problem is characterized by a parameter-dependent, anisotropic rotating velocity field $\mathbf{b}(\mathbf{x}; \boldsymbol{\mu})$. Let the parameter vector be $\boldsymbol{\mu} = (\mu_x, \mu_y) \in \mathcal{D} \subset \mathbb{R}^2$. The velocity field is defined as
\begin{displaymath}
    \mathbf{b}(x, y; \boldsymbol{\mu}) = \begin{bmatrix} \mu_x (y - 0.5) \\ -\mu_y (x - 0.5) \end{bmatrix}.
\end{displaymath}
The parameters $\mu_x$ and $\mu_y$ jointly control the rotation rate and ellipticity of the velocity field. For our numerical experiments, the parameter domain is set to $\mathcal{D} = [15, 20] \times [15, 20]$.

The equation is discretized using continuous piecewise linear Lagrange finite elements on a uniform triangular mesh generated from a $400\times 400$ grid, resulting in $N_h=160{,}801$ degrees of freedom. Since $\varepsilon$ is small and the velocity magnitude is large away from the vortex center, the local Péclet number exceeds one on most elements. We therefore use the same SUPG stabilization strategy as in Case 1:
\begin{displaymath}
    \mathcal{B}_{\mathrm{Gal}}(u, v) + \sum_{K \in \mathcal{T}_h} \int_K \delta_K (\mathbf{b} \cdot \nabla u)(\mathbf{b} \cdot \nabla v)~\mathrm{d}x = \mathcal{L}_{\mathrm{Gal}}(v) + \sum_{K \in \mathcal{T}_h} \int_K \delta_K f (\mathbf{b} \cdot \nabla v)~\mathrm{d}x,
\end{displaymath}
where $\mathcal{B}_{\mathrm{Gal}}$ and $\mathcal{L}_{\mathrm{Gal}}$ denote the standard Galerkin bilinear and linear forms. The stabilization parameter is chosen as
\begin{displaymath}
    \delta_K = \frac{h_K}{2\|\mathbf{b}\|_{\infty,K}} \left( \coth(\mathrm{Pe}_K) - \frac{1}{\mathrm{Pe}_K} \right),\qquad
    \mathrm{Pe}_K = \frac{\|\mathbf{b}\|_{\infty,K} h_K}{2\varepsilon},
\end{displaymath}
where $\|\mathbf{b}\|_{\infty,K}:=\sup_{\mathbf{x}\in K}\|\mathbf{b}(\mathbf{x};\boldsymbol{\mu})\|_2$; when this quantity vanishes, we set $\delta_K=0$.

\begin{table}[t]
\footnotesize
\caption{Results for $100$ random online solves of the anisotropic vortex problem: $n_s=100, L=5$.}\label{tab:anisotropic-vortex}
\begin{center}
\resizebox{\textwidth}{!}{
\begin{tabular}{|c|c|c|c|c|c|}
\hline
KSP Type
& \multicolumn{2}{c|}{GMRES}
& \multicolumn{3}{c|}{FGMRES}\\
\hline

PC Type
& ILU(3)
& AMG
& $\mathbf{M}_{(1)}$
& $\mathbf{M}_{(2)}$
& $\mathbf{M}_{(3)}$\\
\hline

Iterations
& $121.06\pm 4.04$
& $12.14\pm 0.35$
& $3.99\pm 0.10$
& $3.99\pm 0.10$
& $4.00\pm 0.20$\\

\hline

$t_{\mathrm{online}}(\mathrm{s})$
& $1.039\pm 0.019$
& $1.008\pm 0.013$
& $0.627\pm 0.007$
& $0.644\pm 0.011$
& $0.643\pm 0.015$\\

\hline

$t_{\mathrm{offline}}(\mathrm{s})$
& --
& --
& $154.49$
& $156.14$
& $158.23$
\\
\hline
\end{tabular}
 }
\end{center}
\end{table}

\begin{table}[t]
\scriptsize
\caption{Speed-up factors for the anisotropic vortex problem over the $100$ online solves. Each speed-up is computed as $t_{\mathrm{ref}}/t_{\mathrm{ARB}}$ for the same parameter sample.}\label{tab:anisotropic-vortex-speedup}
\begin{center}
\setlength{\tabcolsep}{3pt}
\renewcommand{\arraystretch}{1.05}
\resizebox{\textwidth}{!}{
\begin{tabular}{|c|c|c|c|c|c|c|c|c|c|c|}
\hline
{}
& {}
& \multicolumn{3}{c|}{$\mathbf{M}_{(1)}$}
& \multicolumn{3}{c|}{$\mathbf{M}_{(2)}$}
& \multicolumn{3}{c|}{$\mathbf{M}_{(3)}$}\\
\hline
Parameter range
& Reference solver
& Max
& Mean
& Min
& Max
& Mean
& Min
& Max
& Mean
& Min\\
\hline
$\boldsymbol{\mu}\in[15,20]^2$
& GMRES+ILU(3)
& $1.8\times$
& $1.7\times$
& $1.6\times$
& $1.8\times$
& $1.6\times$
& $1.5\times$
& $1.8\times$
& $1.6\times$
& $1.4\times$\\
\hline
$\boldsymbol{\mu}\in[15,20]^2$
& GMRES+AMG
& $1.7\times$
& $1.6\times$
& $1.5\times$
& $1.6\times$
& $1.6\times$
& $1.5\times$
& $1.7\times$
& $1.6\times$
& $1.4\times$\\
\hline
\end{tabular}
}
\end{center}
\end{table}

The anisotropic vortex problem introduces strong transport and geometric variation through the two velocity parameters. As shown in \cref{tab:anisotropic-vortex}, ILU(3) requires $121.06$ iterations on average and takes about $1.039$ seconds online, whereas AMG reduces the count to $12.14$ and takes about $1.008$ seconds online. The ARB-preconditioned FGMRES methods converge in about $4$ iterations, with online times between $0.627$ and $0.644$ seconds. The speed-up factors in \cref{tab:anisotropic-vortex-speedup} show mean improvements of $1.6\times$--$1.7\times$ over ILU(3) and $1.6\times$ over AMG for all three ARB preconditioners. All minimum speed-ups relative to both ILU(3) and AMG exceed $1\times$, while the mean behavior and robust iteration counts demonstrate a consistent online advantage. The two representative histories in \cref{fig:anisotropic-vortex-residuals} show that this improvement is not merely an averaged effect: for both the worst and the best sampled parameters, the ARB-preconditioned residual falls by several orders of magnitude almost immediately and then reaches the tolerance after a very short Krylov correction phase.

For Type II, the offline time is $156.14$ seconds and the per-solve saving over AMG is $1.008\,\mathrm{s}-0.644\,\mathrm{s}=0.364\,\mathrm{s}$, giving
\begin{displaymath}
    N_{\text{break-even}} = \frac{156.14}{0.364} \approx 429 \text{ solves}.
\end{displaymath}
Thus, after about $429$ parametrized solves, the Type II ARB preconditioner becomes cost-effective for this test.

\begin{remark}
    For this discretization, the diffusion, reaction, and SUPG contributions keep the diagonal entries of $\mathbf{A}_h(\boldsymbol{\mu})$ positive. Hence $\mathbf{T}=\mathbf{D}(\boldsymbol{\mu})=\operatorname{diag}(\mathbf{A}_h(\boldsymbol{\mu}))$ is adopted in the Type III ARB preconditioner.
\end{remark}

\subsection{Case 3: The Stokes equation}
\begin{figure}[ht]
\footnotesize
 \centering
 \subfloat[$\gamma=0.920911$ (the worst case).]{\label{fig:stokes-worst-residual}
  \includegraphics[width=0.45\textwidth]{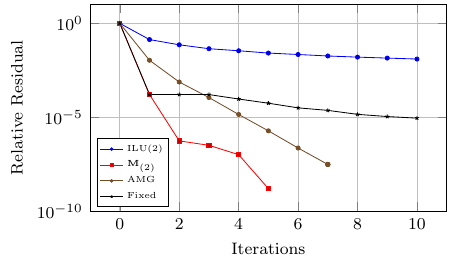}
 }
 \hspace{0.01\textwidth}
 \subfloat[$\gamma=0.142620$ (the best case).]{\label{fig:stokes-best-residual}
  \includegraphics[width=0.45\textwidth]{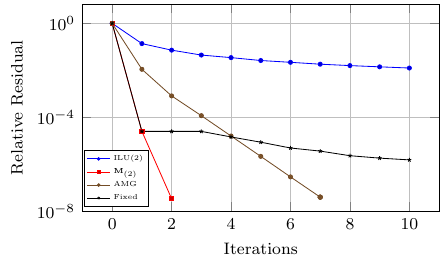}
 }
 \caption{Relative residual norms with different preconditioners for the Stokes equation.}
 \label{fig:stokes-residuals}
\end{figure}
We now consider the Stokes equation, a problem studied extensively in \cite{benzi2005saddle, brezzi1991mixed, silvester1994fast}. 
The system is governed by
\begin{displaymath}
    \begin{cases}
-\nabla \cdot (2\gamma \boldsymbol{\varepsilon}(\mathbf{u})) + \nabla p = \mathbf{f}, & \text{in}~ \Omega, \\
\nabla \cdot \mathbf{u} = 0, &\text{in}~\Omega, \\
\mathbf{u} = \mathbf{0}, &\text{on}~\Gamma_{\!\mathrm{wall}}, \\
\mathbf{u} = \mathbf{u}_{\mathrm{lid}}, &\text{on}~\Gamma_{\!\mathrm{top}},
\end{cases}
\end{displaymath}
where $\mathbf{u}: \Omega\to\mathbb{R}^2$ is the velocity field, $p: \Omega\to\mathbb{R}$ is the pressure, $\gamma>0$ is the dynamic viscosity parameter, $\mathbf{f}=[0, -1]^T$ is the body force representing downward gravitational acceleration, and $\boldsymbol{\varepsilon}(\mathbf{u}) = \frac{1}{2}\left(\nabla\mathbf{u} + \nabla\mathbf{u}^T\right)$ is the strain rate tensor. The domain is the unit square $\Omega = (0,1)\times(0,1)$, whose boundary is partitioned as
\begin{displaymath}
    \Gamma_{\!\mathrm{top}} = \{(x,1): x\in(0,1)\}, \quad
\Gamma_{\!\mathrm{wall}} = \partial\Omega \setminus \Gamma_{\!\mathrm{top}}.
\end{displaymath}
The lid velocity is prescribed as a parabolic profile typical of driven-cavity benchmarks:
\begin{displaymath}
    \mathbf{u}_{\mathrm{lid}}(x) = \begin{bmatrix} U(x) \\ 0 \end{bmatrix}, \quad
U(x) = 4x(1-x).
\end{displaymath}
We use the standard Taylor-Hood finite element pair ($\mathcal{P}_2-\mathcal{P}_1$), which satisfies the inf-sup condition \cite{brezzi1991mixed, elman2014finite}. After imposing the velocity boundary conditions, the discrete saddle-point system has the form
\begin{displaymath}
    \begin{bmatrix}
\mathbf{K} & \mathbf{B}^T \\
\mathbf{B} & \mathbf{O}
\end{bmatrix}
\begin{bmatrix}
\mathbf{u}_h \\
\mathbf{p}_h
\end{bmatrix}
=
\begin{bmatrix}
\mathbf{f}_h \\
\mathbf{0}
\end{bmatrix},
\end{displaymath}
where $\mathbf{K}$ is the viscosity-dependent stiffness matrix and $\mathbf{B}$ is the divergence matrix. The total number of degrees of freedom is $N_h = 362{,}003$, and the viscosity parameter varies over $\gamma\in[0.1,1]$.

The monolithic matrix is symmetric indefinite and has a zero pressure block; consequently, applying BoomerAMG directly to the complete matrix does not respect the saddle-point structure and leads to divergence in the present tests. For a structure-preserving AMG implementation, we use a full Schur-complement field split. BoomerAMG is applied only to the velocity block $\mathbf{K}$, while the pressure Schur operator is preconditioned by the inverse of the viscosity-weighted pressure mass matrix. More precisely, if $\mathbf{M}_p$ denotes the pressure mass matrix, then
\begin{displaymath}
    \mathbf{B}\mathbf{K}^{-1}\mathbf{B}^{T}\sim \gamma^{-1}\mathbf{M}_p,
\end{displaymath}
where $\sim$ denotes spectral equivalence. Hence $(\gamma^{-1}\mathbf{M}_p)^{-1}$ is used for the pressure block, and the constant-pressure nullspace is supplied to the mixed solver. Thus AMG is used only on the elliptic velocity subproblem, with the Schur approximation accounting for velocity--pressure coupling. The AMG entry in \cref{tab:stokes} refers to this structure-preserving field-split implementation.

\begin{table}[t]
\footnotesize
\caption{Results for $100$ random online solves of the Stokes equation: $n_s=30, L=5$.}\label{tab:stokes}
\begin{center}
\resizebox{\textwidth}{!}{
\begin{tabular}{|c|c|c|c|c|c|}
\hline
KSP Type
& \multicolumn{2}{c|}{GMRES}
& \multicolumn{3}{c|}{FGMRES} \\
\hline

PC Type
& ILU(2)
& AMG
& $\mathbf{M}_{(1)}$
& $\mathbf{M}_{(2)}$
& $\mathbf{M}_{(3)}$ \\
\hline

Iterations
& $792.87\pm 51.40$
& $6.78\pm 0.41$
& $3.15\pm 0.83$
& $3.27\pm 0.99$
& $2.21\pm 0.48$ \\
\hline

$t_{\text{online}}(\mathrm{s})$
& $29.312\pm 1.848$
& $4.729\pm 0.423$
& $0.175\pm 0.015$
& $0.178\pm 0.018$
& $0.166\pm 0.013$ \\
\hline

$t_{\text{offline}}(\mathrm{s})$
& --
& --
& $85.58$
& $85.74$
& $86.19$ \\
\hline
\end{tabular}
 }
\end{center}
\end{table}

\begin{table}[t]
\scriptsize
\caption{Speed-up factors for the Stokes equation over the $100$ online solves. Each speed-up is computed as $t_{\mathrm{ref}}/t_{\mathrm{ARB}}$ for the same parameter sample.}\label{tab:stokes-speedup}
\begin{center}
\setlength{\tabcolsep}{3pt}
\renewcommand{\arraystretch}{1.05}
\resizebox{\textwidth}{!}{
\begin{tabular}{|c|c|c|c|c|c|c|c|c|c|c|}
\hline
{}
& {}
& \multicolumn{3}{c|}{$\mathbf{M}_{(1)}$}
& \multicolumn{3}{c|}{$\mathbf{M}_{(2)}$}
& \multicolumn{3}{c|}{$\mathbf{M}_{(3)}$}\\
\hline
Parameter range
& Reference solver
& Max
& Mean
& Min
& Max
& Mean
& Min
& Max
& Mean
& Min\\
\hline
$\gamma\in[0.1,1]$
& GMRES+ILU(2)
& $211.4\times$
& $168.7\times$
& $120.8\times$
& $210.6\times$
& $165.8\times$
& $115.0\times$
& $204.3\times$
& $177.2\times$
& $115.0\times$\\
\hline
$\gamma\in[0.1,1]$
& GMRES+AMG
& $35.1\times$
& $27.2\times$
& $19.0\times$
& $35.0\times$
& $26.8\times$
& $18.1\times$
& $34.0\times$
& $28.6\times$
& $18.1\times$\\
\hline
\end{tabular}
}
\end{center}
\end{table}

\cref{tab:stokes} presents the numerical results for the online tests. In the $100$ random tests, the field-split AMG baseline converged in $6.78$ iterations on average and took $4.729$ seconds online, while ILU(2) required $792.87$ iterations and $29.312$ seconds. The ARB-preconditioned FGMRES methods converged robustly with average iteration counts of $3.15$, $3.27$, and $2.21$ for Types I--III and online times between $0.166$ and $0.178$ seconds. The speed-up statistics in \cref{tab:stokes-speedup} show mean speed-ups of $165.8\times$--$177.2\times$ over ILU(2) and $26.8\times$--$28.6\times$ over AMG; all reported minimum speed-ups remain well above $1\times$. The residual curves in \cref{fig:stokes-residuals} show a gradual decrease for ILU(2) over many Krylov steps, whereas the ARB-preconditioned methods reach the requested tolerance after only a few steps. This behavior is consistent with the RB correction effectively capturing important coupled velocity-pressure components. Thus, for this saddle-point problem, the reduced space is useful as a stabilized preconditioning component.

For the Type II ARB preconditioner, the offline cost is $85.74$ seconds. Its online time is $0.178$ seconds, compared with $4.729$ seconds for AMG, yielding a saving of $4.551$ seconds per solve. The corresponding break-even point is
\begin{displaymath}
    N_{\text{break-even}} = \frac{85.74}{4.729-0.178} \approx 19 \text{ solves}.
\end{displaymath}
Thus, Type II becomes computationally advantageous after about $19$ parametrized solves relative to the AMG baseline.

\begin{remark}
 For Type III, the saddle-point matrix does not have strictly positive diagonal entries because of the pressure block; we therefore set $\mathbf{T}$ to the absolute diagonal of $\mathbf{A}_h(\boldsymbol{\mu})$, replacing entries smaller than $10^{-12}$ by $1$.
\end{remark}

\subsection{Case 4: The Helmholtz equation}
\begin{figure}[ht]
\footnotesize
 \centering
 \subfloat[$k=4.672792$ (the worst case of $k\in [1,5\text{]}$).]{\label{fig:helmholtz-low-wavenumber-residual}
  \includegraphics[width=0.45\textwidth]{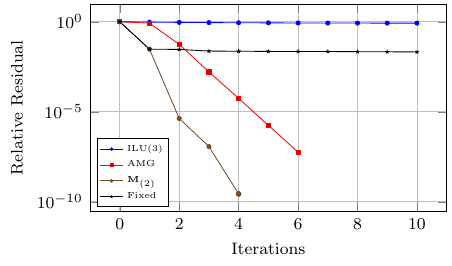}
 }
 \hspace{0.01\textwidth}
 \subfloat[$k=9.960407$ (the worst case of $k\in (5,10\text{]}$).]{\label{fig:helmholtz-high-wavenumber-residual}
  \includegraphics[width=0.45\textwidth]{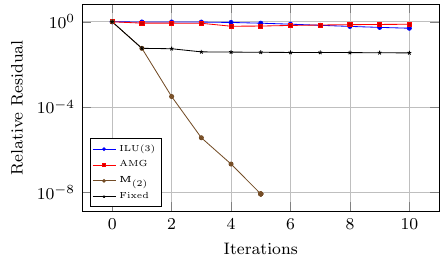}
 }
 \caption{Relative residual norms with different preconditioners for the Helmholtz equation.}
 \label{fig:helmholtz-residuals}
\end{figure}
Finally, we consider the Helmholtz equation, whose high-wave-number finite element discretizations and iterative solvers are known to be challenging \cite{erlangga2008advances, ihlenburg1995finite}:
\begin{displaymath}
    \begin{cases}
    -\Delta u-k^2u=f, &\text{in}~\Omega,\\
    u=0, &\text{on}~\partial\Omega,
\end{cases}
\end{displaymath}
where $u:\Omega\to\mathbb{R}$, $\Omega=(0,1)\times(0,1)$, $f\equiv 1$, and $k\in[1,10]$ is the wave number parameter. We use continuous piecewise linear finite elements on a uniform mesh generated from a $500\times 500$ grid, giving $N_h=251{,}001$ degrees of freedom. The standard finite element discretization yields
\begin{displaymath}
    \mathbf{A}_h(k)\mathbf{u}_h(k)=\mathbf{f}_h,\qquad \mathbf{A}_h(k)=\mathbf{K}-k^2\mathbf{M},
\end{displaymath}
where $\mathbf{K}$ and $\mathbf{M}$ are the stiffness and mass matrices. For this problem, the admissible parameter domain is $\mathcal{D}=\{k\in[1,10]:\mathbf{A}_h(k)\text{ is nonsingular on the working space}\}$. All offline and online samples lie in this domain and thus avoid the discrete resonance values.

\begin{table}[t]
\footnotesize
\caption{Results for the Helmholtz equation. Offline stage: $n_s=50, L=5$ for $k\in [1, 10]$; online stage: $100$ random samples for $k\in [1, 5]$ and $100$ random samples for $k\in (5, 10]$.}\label{tab:helmholtz}
\begin{center}
\resizebox{\textwidth}{!}{
\begin{tabular}{|c|c|c|c|c|c|}
\hline
KSP Type
& \multicolumn{2}{c|}{GMRES}
& \multicolumn{3}{c|}{FGMRES}\\
\hline

PC Type
& ILU(3)
& AMG
& $\mathbf{M}_{(1)}$
& $\mathbf{M}_{(2)}$
& $\mathbf{M}_{(3)}$\\
\hline

Iterations ($k\in[1, 5]$)
& $285.85\pm 410.04$
& $5.69\pm 0.52$
& $3.85\pm 0.36$
& $3.85\pm 0.36$
& $3.57\pm 0.50$\\
\hline

$t_{\mathrm{online}}(\mathrm{s})$
& $1.760\pm 2.442$
& $0.290\pm 0.014$
& $0.097\pm 0.005$
& $0.099\pm 0.005$
& $0.100\pm 0.007$\\
\hline

Iterations ($k\in(5,10]$)
& $803.43\pm 1125.75$
& $18.50\pm 22.97$
& $3.69\pm 0.52$
& $3.71\pm 0.52$
& $3.44\pm 0.52$\\
\hline

$t_{\mathrm{online}}(\mathrm{s})$
& $4.800\pm 6.628$
& $0.596\pm 0.550$
& $0.096\pm 0.007$
& $0.097\pm 0.008$
& $0.099\pm 0.010$\\
\hline

$t_{\mathrm{offline}}(\mathrm{s})$
& --
& --
& $55.78$
& $56.36$
& $56.89$\\
\hline
\end{tabular}
 }
\end{center}
\end{table}

\begin{table}[t]
\scriptsize
\caption{Speed-up factors for the Helmholtz equation based on $100$ online parameter samples in each wave number range. Each comparison includes only samples for which both the reference solver and the ARB solver converge. Each speed-up is computed as $t_{\mathrm{ref}}/t_{\mathrm{ARB}}$ for the same parameter sample.}\label{tab:helmholtz-speedup}
\begin{center}
\setlength{\tabcolsep}{3pt}
\renewcommand{\arraystretch}{1.05}
\resizebox{\textwidth}{!}{
\begin{tabular}{|c|c|c|c|c|c|c|c|c|c|c|}
\hline
{}
& {}
& \multicolumn{3}{c|}{$\mathbf{M}_{(1)}$}
& \multicolumn{3}{c|}{$\mathbf{M}_{(2)}$}
& \multicolumn{3}{c|}{$\mathbf{M}_{(3)}$}\\
\hline
Parameter range
& Reference solver
& Max
& Mean
& Min
& Max
& Mean
& Min
& Max
& Mean
& Min\\
\hline
$k\in[1,5]$
& GMRES+ILU(3)
& $204.6\times$
& $18.5\times$
& $6.8\times$
& $197.3\times$
& $17.9\times$
& $6.4\times$
& $192.8\times$
& $18.1\times$
& $6.3\times$\\
\hline
$k\in[1,5]$
& GMRES+AMG
& $3.6\times$
& $3.0\times$
& $2.7\times$
& $3.4\times$
& $2.9\times$
& $2.5\times$
& $3.3\times$
& $2.9\times$
& $2.5\times$\\
\hline
$k\in(5,10]$
& GMRES+ILU(3)
& $241.0\times$
& $49.1\times$
& $10.0\times$
& $239.5\times$
& $48.9\times$
& $10.0\times$
& $224.4\times$
& $46.9\times$
& $9.9\times$\\
\hline
$k\in(5,10]$
& GMRES+AMG
& $31.5\times$
& $6.1\times$
& $2.9\times$
& $32.6\times$
& $6.1\times$
& $2.9\times$
& $30.1\times$
& $5.9\times$
& $3.1\times$\\
\hline
\end{tabular}
}
\end{center}
\end{table}

Within the tested wave number range, the discrete Helmholtz operator is positive definite for $k^2<\lambda_{1,h}$ and indefinite for $k^2>\lambda_{1,h}$, away from discrete resonances, where $\lambda_{1,h}$ denotes the smallest generalized eigenvalue of $\mathbf{K}\mathbf{q}=\lambda\mathbf{M}\mathbf{q}$ on the homogeneous Dirichlet working space. Indefiniteness and proximity to resonance can make these systems challenging for classical iterative methods \cite{erlangga2006novel, ernst2012difficult}. In the offline stage, we use $n_s=50$ uniformly sampled parameters from $k\in[1,10]$ and construct $L=5$ dynamic reduced spaces. In the online stage, to show the effect of the wave number, we independently test 100 random samples for $k\in[1,5]$ and 100 random samples for $k\in(5,10]$. The statistics in \cref{tab:helmholtz} for ILU(3) and AMG are computed over converged runs only; divergent runs are excluded from these averages. In the $100$ random tests on $k\in[1,5]$, ILU(3) diverged in $2$ cases. On $k\in(5,10]$, AMG diverged in $6$ cases and ILU(3) diverged in $9$ cases. By contrast, all three ARB preconditioners converged for all tested samples, with average iteration counts between $3.44$ and $3.85$ across the two wave number ranges. The speed-up factors in \cref{tab:helmholtz-speedup} confirm that this robustness is also reflected in online time: for $k\in[1,5]$, the mean speed-ups are about $17.9\times$--$18.5\times$ over ILU(3) and $2.9\times$--$3.0\times$ over AMG, while for $k\in(5,10]$ they increase to $46.9\times$--$49.1\times$ over ILU(3) and $5.9\times$--$6.1\times$ over AMG. All reported minimum speed-ups remain above $1\times$ in this test.

The semilog plots in \cref{fig:helmholtz-residuals} illustrate two indefinite cases. Both displayed cases are worst cases in their respective wave number intervals, yet the ARB-preconditioned residual still exhibits a pronounced first-step decrease and then reaches the tolerance in only a few further FGMRES steps. The first drop is consistent with the RB component providing an effective correction for important parameter-dependent error components. The dynamically constructed ARB preconditioner accelerates Krylov convergence relative to ILU(3) or AMG in these tests.

For the Type II ARB preconditioner, the offline cost is $56.36$ seconds. In the range $k\in[1,5]$, the online time of Type II is $0.099$ seconds, compared with $0.290$ seconds for AMG, yielding a saving of $0.191$ seconds per solve. In the range $k\in(5,10]$, the corresponding saving is $0.596\,\mathrm{s}-0.097\,\mathrm{s}=0.499\,\mathrm{s}$ per solve. The break-even points are therefore
\begin{displaymath}
    N_{\text{break-even}} = \frac{56.36}{0.191} \approx 295 \text{ solves}, \qquad
    N_{\text{break-even}} = \frac{56.36}{0.499} \approx 113 \text{ solves}.
\end{displaymath}
Thus, Type II becomes computationally advantageous after about $295$ solves for $k\in[1,5]$ and about $113$ solves for $k\in(5,10]$.

\begin{remark}
The diagonal entries of $\mathbf{A}_h(k)=\mathbf{K}-k^2\mathbf{M}$ are not guaranteed to be positive when the mass term dominates. For Type III, we use the same positive-diagonal substitute as in the Stokes test: the absolute diagonal of $\mathbf{A}_h(k)$, with entries smaller than $10^{-12}$ replaced by $1$.
\end{remark}

\section{Conclusions}\label{sec:conclusions}
In this paper, we have proposed and analyzed a class of ARB preconditioners for accelerating the iterative solution of large-scale parametrized linear systems. The preconditioners are dynamically constructed during the offline stage and applied within the FGMRES iterative solver. Theoretical analysis confirms the nonsingularity of the proposed preconditioners under mild conditions and provides error bounds for the preconditioned Richardson iteration.

Numerical experiments on the convection-diffusion equation, the anisotropic vortex problem, the steady-state Stokes equation, and the Helmholtz equation demonstrate the effectiveness of ARB preconditioners relative to standard preconditioners such as ILU and AMG. The proposed method exhibits robust convergence with respect to parameter variations and achieves significant speed-ups in the online phase. Furthermore, the low break-even points observed in the experiments highlight the practical applicability of ARB preconditioners for multi-query problems that require repeated evaluations of the high-fidelity model.

\section*{Acknowledgments}
The authors are grateful to the anonymous referees for their constructive comments and suggestions that helped improve the quality of this manuscript.

\bibliographystyle{siamplain}
\bibliography{references}
\end{document}